\documentclass[11pt, reqno]{amsart}
\usepackage{amsmath,amssymb,amscd,mathrsfs,amscd}
\usepackage{graphics,verbatim, tikz}

\usetikzlibrary{matrix,arrows,patterns}
\usepackage{amsmath,amstext,amsbsy,amssymb,amscd}
\usepackage{amsmath}
\usepackage{amsxtra}
\usepackage{amscd}
\usepackage{amsthm}
\usepackage{amsfonts}
\usepackage{graphicx}%
\usepackage{amssymb}
\usepackage{eucal}
\usepackage{color}
\usepackage{multirow}
\usepackage[enableskew,vcentermath]{youngtab}
\newtheorem{thm}{Theorem}[section]
\theoremstyle{plain}
\newtheorem{lem}[thm]{Lemma}
\newtheorem{prop}[thm]{Proposition}

\newtheorem{cor}[thm]{Corollary}

\newtheorem{dfn}[thm]{Definition}
\theoremstyle{definition}
\newtheorem{example}[thm]{Example}

\newtheorem{rmk}[thm]{Remark}

\theoremstyle{remark}

\numberwithin{equation}{section}

\newcommand{\A}{\mathcal A}
\newcommand{\Ap}{{\mathcal A^\pi}}

\newcommand{\C}{\mathbb C}
\newcommand{\N}{\mathbb N}

\newcommand{\mfh}{\mathfrak{h}}

\newcommand{\la}{\lambda}

\newcommand{\osp}{\mathfrak{osp}}

\newcommand{\ev}[1]{{#1}_{0}}
\newcommand{\od}[1]{{#1}_{1}}

{\vskip-\lastskip\medskip
  \noindent
  {\em #1.}\enspace
  }%
{\qed\par\medskip
  }

\newcommand{\Q}{{\mathbb{Q}}}
\newcommand{\Z}{{\mathbb Z}}
\newcommand{\Qq}{{\Q(q)}}

\newcommand{\Qqp}{{\Q(q)^\pi}}
\newcommand{\Qp}{{\mathbb{Q}^\pi}}
\newcommand{\Zp}{{\Z^\pi}}

\newcommand{\Zpqqi}{\Zp[q,q^{-1}]}

\newcommand{\bB}{{\mathbf{B}}}
\newcommand{\bF}{{\mathbf{B}}^\bullet}
\newcommand{\cL}{{\mathcal{L}}}

\newcommand{\curlyO}{{\mathcal{O}}}
\newcommand{\curlyP}{{\mathcal{P}}}

\newcommand{\curlyC}{{\mathcal{C}}}

\newcommand{\ff}{{\bf f}}

\newcommand{\UU}{{\bf U}}
\newcommand{\UUZ}{{}_\Z{\bf U}}
\newcommand{\Uqp}{\UU_{q,\pi}}

\newcommand{\bbinom}[2]{\begin{bmatrix}#1 \\ #2\end{bmatrix}}
\newcommand{\cbinom}[2]{\set{\!\!\!\begin{array}{c} #1 \\ #2\end{array}\!\!\!}}
\newcommand{\abinom}[2]{\ang{\!\!\!\begin{array}{c} #1 \\ #2\end{array}\!\!\!}}
\newcommand{\qfact}[1]{[#1]^!}

\newcommand{\ir}{{\vphantom{|}_i r}}
\newcommand{\ri}{{r_i}}

\newcommand{\End}{{\operatorname{End}}}

\newcommand{\Rad}{{\operatorname{Rad}}}

\newcommand{\im}{{\operatorname{im}}}
\newcommand{\diag}{{\operatorname{diag}}}

\newcommand{\height}{{\operatorname{ht}}}

\newcommand{\sgn}{{\operatorname{sgn}}}

\newcommand{\catO}{{\mathcal{O}}}
\newcommand{\catOint}{{\mathcal{O}^{\rm int}}}

\newcommand{\zero}{{\bar{0}}}
\newcommand{\one}{{\bar{1}}}

\newcommand{\parity}[1]{{p(#1)}}

\newcommand{\Ad}{{\rm{Ad}}}
\newcommand{\tK}{{\widetilde{K}}}
\newcommand{\tJ}{{\widetilde{J}}}

\newcommand{\te}{{\tilde{e}}}
\newcommand{\tf}{{\tilde{f}}}

\newcommand{\set}[1]{\left\{#1\right\}}
\newcommand{\parens}[1]{\left(#1\right)}
\newcommand{\ang}[1]{\left\langle#1\right\rangle}

\renewcommand{\bar}[1]{\overline{#1}}

\begin{document}
\title{\textsc{Quantum supergroups II. Canonical basis}}

\author[Clark, Hill, Wang]{Sean Clark, David Hill, and Weiqiang Wang}
\address{Department of Mathematics, University of Virginia, Charlottesville, VA 22904}
\email{sic5ag@virginia.edu (Clark), \quad deh4n@virginia.edu (Hill), \quad ww9c@virginia.edu (Wang)}


\begin{abstract}
Following Kashiwara's algebraic approach,
we construct crystal bases and canonical bases for quantum supergroups of anisotropic type and for their integrable modules.
\end{abstract}


\maketitle

\setcounter{tocdepth}{1}
 \tableofcontents

\section{Introduction}

The theory of canonical bases for quantum groups and their integrable modules was introduced by Lusztig
and subsequently by Kashiwara  through a different crystal basis approach  (see \cite{Lu1, Lu2, Ka}).
Among many applications, the
canonical bases have recently played an important role in categorification.

For quantum supergroups, there have been some combinatorial
constructions of crystal bases. For quantum $\osp(1|2n)$, these were constructed for the integrable modules
in \cite{Zou} and in Musson-Zou \cite{MZ};
also see Jeong \cite{Jeo} for a generalization which we will discuss below.  For quantum $\mathfrak{gl}(m|n)$,  crystal bases
for the polynomial representations were  obtained by Benkart-Kang-Kashiwara \cite{BKK}.
More recently, crystal bases for a class of infinite-dimensional simple modules of quantum $\osp(r|2n)$
have been constructed in \cite{Kw}. 
However, none of the authors constructed a crystal basis for
the negative part of these quantum supergroups.
The conventional wisdom among experts seemed to favor
the non-existence of canonical basis (or global crystal basis) for quantum supergroups -- until our recent announcement  \cite{HW, CW}.

The goal of this paper is to systematically develop a theory of canonical bases for half a quantum supergroup
and the associated integrable modules for the first time.
In the super setting Lusztig's geometric approach  is not applicable directly.
Instead, we follow Kashiwara's algebraic approach \cite{Ka} as a blueprint for this paper.

The class of quantum supergroups $\UU$ considered in this paper is of {\em anisotropic type} 
(which means no isotropic odd simple roots)
and satisfies an additional bar-consistent condition (see \cite{CHW} for a foundation of  
quantum supergroups of anisotropic type in which we generalized Lusztig \cite[Chapter~1]{Lu2}; also see earlier related
work \cite{Ya, BKM}). This class of anisotropic type quantum supergroups 
includes quantum $\osp(1|2n)$ as the only finite type examples.
As noted by two of the authors \cite{HW},  it is conceptual to introduce what we call quantum covering groups
with an additional formal parameter $\pi$ with $\pi^2=1$ in place of a super sign for this class of quantum supergroups, and
the quantum covering groups (and respectively, the quantum supergroups) afford a novel bar involution which sends $q\mapsto \pi q^{-1}$
(and respectively, by a specialization $\pi \mapsto -1$).
Our work \cite{HW, CW} was motivated by \cite{EKL, KKT, Wa} and in turn led to new development 
in categorification of quantum covering and super groups \cite{KKO, EL}.

Kashiwara's approach starts with developing a combinatorial theory of crystal basis, which is roughly speaking a basis at $q=0$.
An earlier paper of Jeong \cite{Jeo} constructs crystal bases for integrable modules $V(\la)$ for (a variant of)
quantum supergroups of anisotropic type. However, our work
differs from {\em loc. cit.} in several aspects. First, we deal with the integrable modules $V(\la)$ for all dominant
integral weights $\la \in P^+$ while Jeong put a restriction on $\la$ to a subset of ``even" dominant weights
(as inherited from \cite{Kac, BKM}). Some extra generators $J_\mu$   of $\UU$ introduced in \cite{CHW} 
(which was inspired by \cite{CW}) make the constructions
of $V(\la)$ over $\Q(q)$ for all $\la\in P^+$ possible. Secondly, the odd rank one quantum $\osp(1|2)$ \cite{CW} admits a $2$-dimensional simple module $V(1)$
(here $1$ is an ``odd" weight unavailable in \cite{Jeo}),
which plays a basic role in developing the tensor product rule of crystal bases.
The tensor product rule in the super setting as developed in \cite{Jeo} requires some fixing of super signs,
and the proof therein is complicated for lack of this $2$-dimensional module.
(There is also an earlier version of tensor product rule in \cite{MZ}, where the authors had to work with modules over $\C(q)$ instead of $\Q(q)$.)

Here is the layout of the paper, and we mention explicitly when Kashiwara's approach requires more
noticeable modifications along the way.

In Section~\ref{sec:supergroups}, we set up the notations of quantum
covering/super groups and provide a quick review of the
basics as developed in \cite{CHW}.  Two coproducts of $\UU$ (differing from each other by some $J_\mu$ operators),
both corresponding to the one used in Kashiwara's approach, are introduced here.
In Section~\ref{sec:qpiboson}, a $(q,\pi)$-boson algebra is introduced in order to formulate the crystal basis for $\UU^-$, and
its basic properties are established. In particular, we introduce a bilinear form (called polarization) on $\UU^-$ and
Kashiwara's operators on $\UU^-$.   

In Section~\ref{sec:CBpolar}, we formulate the notion of crystal lattice and crystal basis of integrable modules suitable in the super setting,
a variant of which goes back to \cite{BKK} (for polynomial representations of quantum $\mathfrak{gl}(m|n)$).
 A polarization on an integrable module is formulated, and we note an unusual
super phenomenon of polarization on a tensor product of two integrable modules.
In addition, we establish the tensor product rule of crystal basis in which  $\pi$ appears.
We formulate the main theorems of crystal bases parallel to Kashiwara's.
In Section~\ref{sec:grandloop}, we adapt Kashiwara's grant loop inductive argument to prove the main theorems of crystal bases.

In Section~\ref{sec:proppol}, we study further properties of polarization. We show that the crystal basis is $\pi$-orthonormal
in the sense of Definition~\ref{dfn:pibasis},
but in general not orthonormal in the usual sense, with respect to the polarization at $q=0$; in particular, the polarization at $q=0$
is not positive definite. This leads to a key difference in the super setting that neither a crystal lattice nor signed crystal basis in general affords
an orthonormality characterization. In the usual quantum group setting such an orthonormality was established by Lusztig and Kashiwara  \cite{Lu1, Lu2, Ka},
and it readily implies another fundamental fact that the crystal lattice $\cL(\infty)$ on $\UU^-$ is preserved by
the anti-involution $\varrho$ which fixes each Chevalley generator $F_i$.
In the super setting, it continues to be true that
 the crystal lattice $\cL(\infty)$ is $\varrho$-stable. 
The proof of this fact is postponed to \cite{CFLW}, as it requires tools somewhat outside the setting of this paper
and it follows most readily from the connection between $\UU$ and the usual quantum groups developed in {\em loc. cit.}
(This connection is closely related to a remarkable connection between
$2$-parameter quantum groups and the usual quantum groups developed by Fan and Li \cite{FL}.)

In Section~\ref{sec:CB}, we establish the existence of canonical bases for $\UU^-$
and all integrable modules $V(\la)$ with $\la \in P^+$.

By $\UU$-module  in this paper we always mean
a $\UU$-module with a $\Z_2$-grading which is compatible with the action of
the superalgebra. By $\UU$-homomorphism, we always mean a $\Z_2$-graded linear map whose
$\Z_2$-homogeneous parts supercommute with the $\UU$-action.

\vspace{.3cm}

\noindent {\bf Acknowledgement.}
Both first and third authors gratefully acknowledge the
support and stimulating environment at the Institute of Mathematics, Academia Sinica, Taipei, during their visits in Spring 2013.
The first author is partially supported by the Semester Fellowship from Department of Mathematics, University of Virginia (UVA).
The third author is partially supported by NSF DMS-1101268 and the UVA Sesqui Fellowship.

\vspace{.3cm}


\section{Quantum covering and super groups}
\label{sec:supergroups}

\subsection{Super generalized Cartan matrix (SGCM)}
\label{subsec:SGCM}

Let $\pi$ and $q$ be formal indeterminants such that $\pi^2=1$.
For a subring $R$ of the rational function field $\Qq$, define a new ring
$$
R^\pi =R\otimes_\Z \Z[\pi].
$$
 We note the following properties of this ring:
$R^\pi$ is a subring of $\Qqp$; $x\in R^\pi$ is a zero divisor if and only if $x=r(\pi\pm 1)$ for some $r\in R$.

For $n\in\Z$ and $a\in\N$, we define the  {\em $(q,\pi)$-integer}
\begin{equation*}
 [n]=\frac{(\pi q)^n-q^{-n}}{\pi q - q^{-1}}\in \Zpqqi,
\end{equation*}
and then define the corresponding $(q,\pi)$-factorials and $(q,\pi)$-binomial coefficients
for $a\in \N$ by
\begin{equation*}
\qfact{a}=\prod_{i=1}^a [i],\qquad
\bbinom{n}{a}=\frac{\prod_{i=1}^a[n+i-a]}{\qfact{a}}.
\end{equation*}
We adopt the convention that $\qfact{0}=1$. Note that
$\bbinom{n}{a}=\frac{\qfact{n}}{\qfact{a}\qfact{n-a}}$, for  $n\geq
a\geq 0$.
There is a unique $\Qp$-linear map
\begin{equation}\label{eqn:barinv}
\psi:\Qqp\longrightarrow \Qqp,\quad q\mapsto \pi q^{-1}.
\end{equation}
We will also use the notation
$\bar{f(q)}:=\psi(f(q))$ and call this map the bar involution.
Note that the induced map on $\Qqp/\ang{\pi-1} \cong \Qq$ is the usual
bar involution.

Let $I=\ev I \coprod \od I$ such that $\od I\neq \emptyset$.
We define the parity function
$\parity{\cdot}: I\rightarrow \Z_2$
via $\parity{i}= t$ if $i\in I_{t}$.

Assume $|I|=\ell$.
We call a matrix $A=(a_{ij})_{i,j\in I}$ a super generalized
Cartan matrix (SGCM) of {\em anisotropic type}
if $A$ satisfies the following conditions:

\begin{enumerate}
    \item[(a)] $a_{ii}=2$ for $i\in I$;
    \item[(b)] $a_{ij}\in -\N$ for $i\neq j\in I$;
    \item[(c)] $a_{ij}=0$ if and only if $a_{ji}=0$ for $i,j\in I$;
    \item[(d)] $a_{ij}\in 2\Z$ if $i\in \od I$ and $j\in I$;
    \item[(e)] there exists an invertible diagonal matrix
            $D=\diag(d_i:i\in I)$ with
            $d_i\in \Z_{> 0}$ and
            $\gcd(d_i:i\in I)=1$
            such that $DA$ is symmetric.
\end{enumerate}
We associate to a SGCM of anisotropic type the following data:
\begin{itemize}
\item a finite dimensional $\Q$-vector space $\mfh$;
\item linearly independent subsets $\set{\alpha_i\in \mfh^*|i\in I}$ and
        $\set{\alpha^\vee_i\in \mfh|i\in I}$ such that
        \[\ang{\alpha^\vee_i,\alpha_j} =a_{ij};\]
        (Here and below we denote the natural pairing between $\mfh$ and $\mfh^*$ by $\ang{\cdot,\cdot}$.)
\item a lattice $P$ in $\mfh^*$ containing $\alpha_j \; (j\in I)$ such that $\ang{\alpha^\vee_i, \la} \in \Z \; (i\in I, \la \in P)$.
\end{itemize}
We define the lattice $P^\vee:=\set{h\in \mfh | \ang{h,P} \subseteq \Z}$.
We have a sublattice $Q :=\bigoplus_{i\in I} \Z \alpha_i\subset P$ and
a lattice $Q^\vee :=\bigoplus_{i \in I} \Z \alpha^\vee_i\subset P^\vee$.
Let $Q^+=\bigoplus_I \Z_{\geq 0} \alpha_i$
and $Q^-=-Q^+$.
Denote $P^+ =\{\la \in P \mid \ang{\alpha_i^\vee,\la} \in \Z_{\geq 0}, \forall i \in I\}$.  Fix
 $\omega_j \in P^+$ for each $j\in I$ such that $\ang{\alpha_i^\vee,\omega_j} =\delta_{ij}$ for $i\in I$.

SGCM's of anisotropic type were first introduced in \cite{Kac} though the terminology ``anisotropic type" is new
(which means that there are no isotropic odd simple roots),
and Condition (d) above ensures that  all the odd simple roots are of type $\osp(1|2)$.

For $i\in I$ set \[q_i=q^{d_i}, \qquad \pi_i=\pi^{\parity{i}}, \qquad [n]_i=\frac{(\pi_iq_i)^n-q_i^{-n}}{\pi_iq_i-q_i^{-1}},\]
and so on.
Throughout this paper,
we will always assume that a SGCM is of anisotropic type satisfying the following additional {\em bar-consistent} assumption:
\begin{enumerate}
\item[(f)]  $\parity{i}=d_i$ mod $2$.
\end{enumerate}
This assumption immediately implies $\bar{[n]_i}=[n]_i$, and is indispensable in formulating the
bar involution and canonical basis which is the main goal of this paper.
We note, however, that this assumption (f) is unnecessary for the definition
of the quantum cover groups given in \cite{CHW} (see \S \ref{subset:covering} below) and the
$(q,\pi)$-boson algebra introduced in Section~\ref{sec:qpiboson}.

\subsection{Quantum covering group}
 \label{subset:covering}

Let $A$ be a SGCM. The quantum covering group associated to $A$ is defined
to be the associative $\Qqp$-superalgebra $\Uqp(A)$ (with $1$)
generated by
\[\ E_i,\quad F_i,\quad J_\mu,
\quad K_\mu \quad (i\in I, \mu\in P^\vee),   \]
subject to the relations \eqref{eq:relK}-\eqref{eq:relEF}  below:
\begin{align}  \label{eq:relK}
K_0=1, \quad K_\mu K_\nu=K_{\mu+\nu},   & \quad \text{ for }\mu,\nu\in P^\vee, i\in I;
 \\
\label{eq:relJ}
 J_{2\mu}=1, \quad J_\mu J_\nu=J_{\mu+\nu},  & \quad \text{ for }\mu,\nu\in P^\vee, i\in I;
 \\
J_\mu K_\nu=K_\nu J_\mu,  & \quad \text{ for }\mu,\nu\in P^\vee;
 \\
K_\mu E_i=q^{\ang{\mu,\alpha_i}}E_iK_{\mu},\quad
 J_\mu E_i =\pi^{\ang{\mu,\alpha_i}}E_iJ_{\mu},   & \quad  \text{ for }i\in I, \mu\in P^\vee;
\label{eq:relKE+JE}
 \\
\label{eq:relKF+JF}
K_\mu F_j=q^{-\ang{\mu,\alpha_j}}F_jK_{\mu},
\quad J_\mu F_i =\pi^{-\ang{\mu,\alpha_i}}F_iJ_{\mu},  & \quad \text{ for }i,j\in I, \mu\in P^\vee;
\\
\sum_{s+t=1-a_{ij}}(-1)^{s}\pi_i^{sp(j)+\binom{s}{2}}
\bbinom{1-a_{ij}}{s}_iE_i^{t}E_jE_i^{s}=0,   & \quad \text{ for }i\neq j\in I;
 \\
\sum_{s+t=1-a_{ij}}(-1)^{s}\pi_i^{sp(j)+\binom{s}{2}}
\bbinom{1-a_{ij}}{s}_iF_i^{t}F_jF_i^{s}=0,  & \quad \text{ for }i\neq j\in I;
 \\
\label{eq:relEF}
 E_iF_j-\pi^{\parity i \parity j} F_jE_i=
\delta_{ij}\frac{\tJ_i\widetilde{K}_i-\widetilde{K}_{-i}}{\pi_iq_i-q_i^{-1}},
&  \quad \text{ for }i,j\in I.
\end{align}

In the above, we use the notations $K_i:=K_{\alpha^\vee_i}$ and 
$J_i:=J_{\alpha^\vee_i}$ for $i\in I$, and
$\widetilde{K}_\nu=\prod_i K_{i}^{d_i\nu_i}$
and $\tJ_\nu = \prod_i J_{i}^{d_i\nu_i}$ for any
$\nu=\sum_i \nu_i \alpha^\vee_i\in Q^\vee$.
In particular, 
$$
\widetilde{K}_i=K_{i}^{d_i}, \qquad \tJ_i=J_{i}^{d_i}. 
$$
The superalgebra structure on $\Uqp$ is given by
letting $E_i, F_i$ for $i\in \od I$ be odd and the other generators be even.
In particular, $\Uqp(A)$ inherits a $\Z_2$-grading from $I$:
\begin{equation}\parity{E_i}=\parity{F_i}=\parity{i},\quad
\parity{K_\mu}=\parity{J_\mu}=0, \qquad (i\in I, \mu \in P^\vee).
\end{equation}


\begin{rmk}\label{rmk:compareCW}
Let $I=\set{i}=I_1$ and $Q=P^\vee =\Z \alpha^\vee_i$, and so $A$ is the SGCM associated to the Lie superalgebra
$\osp(1|2)$. Then $\Uqp(A)$ specialized at $\pi=-1$ is a variant of the quantum $\osp(1|2)$ $U_0\oplus U_1$
defined in \cite{CW}. That is,
$\Uqp(A)/\ang{\pi+1, J_i-1} \cong U_0$ and $\Uqp/ \ang{\pi+1, J_i+1} \cong U_1$.
\end{rmk}

Henceforth, we fix a   SGCM $A$ and use the shorthand notation $\UU =\Uqp(A).$ 
The algebra $\UU$ has a triangular decomposition
$\UU\cong \UU^-\otimes \UU^0 \otimes \UU^+$ where $\UU^-$ (resp. $\UU^+$, $\UU^0$)
is the subalgebra generated by $\set{F_i \mid i\in I}$ (resp.
$\set{E_i\mid i\in I}$, $\set{K_\mu,J_\mu \mid  \mu\in P^\vee}$); see \cite{CHW}.
For $i\in I$, and $n\geq 1$ define the divided powers
$E_i^{(n)}=E_i^n/[n]_i!,  F_i^{(n)}=F_i^n/[n]_i!$.
The following fact will be used often later on.

\begin{lem}
For each $k\in I$, $\tJ_k$ is central in $\UU$.
\end{lem}

\begin{proof}
It suffices to check that $\tJ_k$ commutes with every $E_i$ and $F_i$. By \eqref{eq:relKE+JE} and \eqref{eq:relKF+JF}
and using $\tJ_i=J_{i}^{d_i}$,
this is equivalent to checking that $\pi^{d_ka_{ki}}=1$ for all $k,i \in I$,
 which follows from Conditions~(d) and (f) in \S \ref{subsec:SGCM}.
\end{proof}


%

\subsection{Coproducts}

Let $\UU\otimes \UU$ be the space $\UU\otimes_\Qqp \UU$ with the usual
$\Z_2$-grading. We endow $\UU\otimes \UU$ with a multiplication given
by
\begin{equation}\label{eqn:STM}
(a\otimes b)(c\otimes d)=\pi^{p(b)p(c)}(ac\otimes bd)
\end{equation}
for homogeneous $a,b,c,d\in \UU$. This makes $\UU\otimes \UU$
into an associative $\Qqp$-superalgebra with $1$.

We have a superalgebra homomorphism $\Delta:\UU\longrightarrow \UU\otimes \UU$ given by
\begin{align*}
\Delta(E_i) &=E_i\otimes \tK_i^{-1}+\tJ_{i}\otimes E_i,\\
\Delta(F_i) &=F_i\otimes 1 + \tK_i\otimes F_i,\\
\Delta(K_\mu)& =K_\mu\otimes K_\mu,\\
\Delta(J_{\mu}) &=J_{\mu}\otimes J_{\mu},
\end{align*}
where $i\in I, \mu \in P^\vee$.
This homomorphism is in fact a coproduct which makes $\UU$ a Hopf superalgebra in a suitable
sense (a different coproduct was used in \cite{CHW} which is compatible with \cite{Lu2} rather than with \cite{Ka}). We have
\begin{equation}
\Delta(E_i^{(p)})=\sum_{p'+p''=p}
q_i^{p'p''} \tJ_i^{(p'')}E_i^{(p')}\otimes \tK_i^{-p'} E_i^{(p'')},
\end{equation}
\begin{equation}\label{eqn:deltadpf}
\Delta(F_i^{(p)})=\sum_{p'+p''=p}
(\pi_i q_i)^{-p'p''} F_i^{(p')}\tK_i^{p''}\otimes F_i^{(p'')}.
\end{equation}
For $\UU$-modules $M$ and $N$, we endow the tensor product
$M\otimes N$ with a $\UU$-module structure induced by $\Delta$.

There is another coproduct
\begin{equation}  \label{eq:comult'}
\Delta':\UU\longrightarrow \UU\otimes \UU
\end{equation}
given by
\begin{align*}
\Delta'(E_i) &=E_i\otimes \tK_i^{-1}+1\otimes E_i,\\
\Delta'(F_i) &=F_i\otimes 1 + \tJ_i\tK_i\otimes F_i,\\
\Delta'(K_\mu) &=K_\mu\otimes K_\mu,\\
\Delta'(J_{\mu}) &=J_{\mu}\otimes J_{\mu}.
\end{align*}
This induces a second $\UU$-module structure on the tensor product of $M$ and $N$, and we denote
this $\UU$-module by $M\otimes' N$.

\begin{rmk}\label{rem:coproducts}
In case $I=I_0$,  we have $\Delta=\Delta'$, and this coproduct coincides with $\Delta_-$
in \cite[(1.4.3)]{Ka} used in the construction
of crystal bases therein.
In the general super case, there is no compelling reason to choose one of these coproducts
over the other. In fact both $\Delta$ and $\Delta'$ will be useful, as we shall see in Lemma \ref{lem:pseudopol}.
\end{rmk}

\subsection{Module categories}
\label{subsec:cat}

The {\em specialization at $\pi=1$} (respectively, at {\em
$\pi=-1$}) of a $\Qq^\pi$-algebra $R$ is understood as
$\Qq\otimes_{\Qq^\pi}R$, where $\Qq$ is the $\Qq^\pi$-module with
$\pi$ acting as $1$ (respectively, as $-1$).
The specialization at $\pi=1$ of
the algebra $\UU$, denoted by $\UU_+$, is a variant of this quantum group, with some
extra central elements $J_\mu$ that act trivially on all representations (see below).
The specialization at $\pi=-1$ of $\UU$, denoted by $\UU_-$,  is a {\em quantum supergroup}, and the extra
generators $J_\mu$ allow us to formulate integrable modules $V(\la)$
for {\em all} $\la \in P^+$, which was not possible before; see \cite{CHW}.

 In the remainder of this paper, by a representation of the algebra $\UU$
 we mean a  $\Qq^\pi$-module on which $\UU$ acts.
 Note we have a direct sum decomposition of the $\Qq^\pi$-module
 $\Qq^\pi \cong \Qqp/ \ang{\pi-1} \oplus   \Qqp/\ang{\pi+1}$, where $\pi$ acts as $\pm 1$
on $\Qqp/\ang{\pi\mp 1} \cong \Qq$.

We define the category $\curlyC$ (of weight $\UU$-modules) as follows. An object of $\curlyC$
is a $\Z_2$-graded $\UU$-module $M=M_\zero \oplus M_\one$,
compatible with the $\Z_2$-grading on $\UU$, with a given weight
space decomposition
\[
M=\bigoplus_{\lambda\in X} M^\lambda,
 \qquad M^\lambda=\set{m\in M \mid
K_\mu m=q^{\ang{\mu, \lambda}} m, J_\mu m=\pi^{\ang{\mu, \lambda}}
m, \forall \mu \in P^\vee},
\]
such that $M^\lambda=M_\zero^\lambda\oplus M_\one^\lambda$ where
$M_\zero^\lambda=M^\lambda\cap M_\zero$ and
$M_\one^\lambda=M^\lambda\cap M_\one$. The $\Z_2$-graded structure
is only particularly relevant to tensor products, and will generally
be suppressed when irrelevant.
%
We have the following $\Qq^\pi$-module decomposition for each weight space:
$M^\la \cong  M^\la/\ang{\pi-1} \oplus   M^\la/\ang{\pi+1}$; accordingly, we have
$M\cong M_+ \oplus M_-$ as $\UU$-modules, where
$M_\pm :=\oplus_{\la\in X}  M^\la/\ang{\pi \mp 1}$ is an $\UU$-module on which $\pi$ acts as $\pm 1$,
i.e. a $\UU|_{\pi=\pm 1}$-module.
Hence the category $\curlyC$ decomposes into a direct sum
$\curlyC = \curlyC_+ \oplus \curlyC_-$, where $\curlyC_\pm$ can be identified with categories
of weight modules over the specializations $\UU_\pm$.

The BGG category $\catO$ and the category $\catO^{\rm int}$ of integrable modules can be defined as usual (cf. \cite[\S2.5-2.6]{CHW}).
The decomposition $\curlyC = \curlyC_+ \oplus \curlyC_-$ gives rise
to a similar decomposition $\catO^{\rm int} = \catO^{\rm int}_+ \oplus \catO^{\rm int}_-$ thanks to $\catO^{\rm int} \subset\catO \subset \curlyC$.

Let  $I(\lambda)$ be the left
ideal of $\UU$ generated by $\{E_i, K_\mu-q^{\ang{\mu,\lambda}},
J_\mu-\pi_i^{\ang{\mu,\lambda}}\mid \mu\in P^\vee \}$, for each $\lambda\in P$, and define
the Verma module
associated to $\lambda$ to be $M(\lambda):=\UU/I(\lambda)$. Then we have
$M(\la) =M(\la)_+ \oplus M(\la)_-$ and each $M(\la)_\pm$ has a unique
quotient $\UU$-module $V(\la)_\pm$.
Note that $V(\lambda)_\pm^\lambda=\Qq v^+_\lambda$.
We emphasize that in our definition, $\parity{v^+_\lambda}=0$.
The following was proved in \cite{CHW}.

\begin{prop}  \label{prop:integrable}
\cite[Theorem 3.3.2, Corollary 3.3.3]{CHW}
The category $\catOint$ is semisimple; $\set{V(\lambda)_{\pm} \mid \lambda\in P^+}$
            forms a set of pairwise non-isomorphic simple modules in $\catOint$.
\end{prop}

For $M\in \mathcal \curlyC$ and $m\in M_\mu$, we write $|m|=\mu$
and call this the weight grading.
In particular, the elements of $\UU$ have a weight grading given by
$|K_\mu|=|J_\mu|=0, |E_i|=\alpha_i,  |F_i|=-\alpha_i \;\; (i\in I, \mu\in P^\vee).$


\begin{rmk}  \label{rmk:char=}
The characters of $V(\la)_+$ and $V(\la)_-$ coincide for all $\la \in P^+$.
(This was stated in \cite{CHW}, and a proof is given in \cite{KKO} and \cite{CFLW}.)
Thus, $V(\la) :=M(\la)/\big \langle F_i^{\ang{\alpha_i^\vee, \la}+1}v_\la^+\mid i\in I \big\rangle$
is free as a $\Qqp$-module, and  $V(\la) \cong V(\la)_+ \oplus V(\la)_-$.
\end{rmk}


\section{The $(q,\pi)$-boson superalgebra}
\label{sec:qpiboson}

In this section, we formulate a $(q,\pi)$-version of the $q$-Boson algebra as found in Kashiwara \cite[\S 3]{Ka}.
Note that when $\pi$ is specialized to $1$, this is just
the $q$-Boson algebra therein.

\subsection{The algebra $\bF$}

Let $\bF =\bF_{q,\pi}$ be the $\Qqp$-superalgebra generated by odd elements $e, f$
subject to the relation
\[
ef=\pi q^{-2} fe +1.
\]
We set $f^{(n)}=f^n/[n]!$.

One checks that
\begin{equation}
    e^nf^{(m)}=\sum_{t\ge 0}  (\pi q)^{\binom{t+1}{2}-nm}q^{-(n-t)(m-t)}
    \bbinom{n}{t} f^{(m-t)} e^{n-t}.
\end{equation}

The following properties may be directly verified.

\begin{lem}\label{lem:Fprops}
Let $M$ be a $\bF$-module which is locally finite for $e$.
\begin{enumerate}
\item[(i)]
$P=\sum_{n\geq 0} (-1)^n q^{-\binom{n}{2}}f^{(n)}e^n$
defines an endomorphism of $M$ satisfying
\begin{equation}\label{eq:Fprojection}
    eP=Pf=0\text{ and }
    \displaystyle\sum_{t\geq 0} (\pi q)^{\binom{t}{2}}f^{(t)} P e^t=1.
\end{equation}

\item[(ii)]
Let $m\in M$. Then any $u\in M$ has a unique decomposition
$u=\sum_{n \ge 0}  f^{(n)} u_n$ where $u_n\in \ker e$; in fact,
$u_n=(\pi q)^{\binom{n}{2}}P e^n u$.

\item[(iii)]
$M=\im f\oplus \ker e$. Moreover $P:M\rightarrow M$ is the projection map
onto $\ker e$ along this direct sum decomposition.
\end{enumerate}
\end{lem}

\subsection{The algebra $\bB$}

Define the algebra $\bB'$ to be the $\Qqp$-algebra
generated by the elements $\set{e_i,f_i | i\in I}$
subject to the relations

\begin{equation}\label{eq:bP'defreln}
e_if_j=\pi^{p(i)p(j)}q_i^{-a_{ij}}f_je_i+\delta_{ij}, \quad \text{ for all } i,j\in I.
\end{equation}
Then $\bB'$ is naturally a superalgebra  with parity on generators given by
$\parity{e_i}=\parity{f_i}=\parity{i}$ for $i\in I$.
Set  $f_i^{(n)}=f_i^n/[n]_i!$.
The superalgebra $\bB$ by definition has the same generators as $\bB'$
subject to the relation \eqref{eq:bP'defreln} and the additional $(q,\pi)$-Serre relations \eqref{eq:Serre}:
\begin{align}  \label{eq:Serre}
 \begin{split}
\sum_{t=0}^{b_{ij}} (-1)^t \pi_i^{\binom{t}{2}+tp(j)}
\bbinom{b_{ij}}{t}_i e_i^{b_{ij}-t}e_je_i^{t} &=0,
\\
\sum_{t=0}^{b_{ij}} (-1)^t \pi_i^{\binom{t}{2}+tp(j)}
\bbinom{b_{ij}}{t}_i f_i^{b_{ij}-t}f_jf_i^{t} &=0,
\end{split}
\end{align}
where  we denote
$$b_{ij}=1-a_{ij}.
$$
Let
\begin{equation}  \label{eq:Sij}
S_{ij}=\sum_{t=0}^{b_{ij}} (-1)^t \pi_i^{\binom{t}{2}+tp(j)}
\bbinom{b_{ij}}{t} e_i^{b_{ij}-t}e_je_i^{t} \in \bB'.
\end{equation}

\begin{lem}\label{particleserre}
The following holds in $\bB'$ for all $i,j,k\in I$ with $i\neq j$:
\[
S_{ij}f_k=\pi_k^{b_{ij}p(i)+p(j)}q_k^{-\ang{\alpha_k^\vee, b_{ij}\alpha_i+\alpha_j}}f_kS_{ij}.
\]
\end{lem}

\begin{proof}
Let $C_{ij}^k=S_{ij}f_k-
\pi_k^{b_{ij}p(i)+p(j)}q_k^{-\ang{\alpha_k^\vee, b_{ij}\alpha_i+\alpha_j}}f_kS_{ij}.
$

If $k\neq i,j$ then then it is apparent that $C_{ij}^k=0$ from the defining
relations.

When $k=j$, then we have
\begin{equation*}
C_{ij}^j
=e_i^{b_{ij}}\sum_{t=0}^{b_{ij}} (-1)^t q_i^{-ta_{ij}}
\pi_i^{\binom{t}{2}}\bbinom{b_{ij}}{t}_i=0,
\end{equation*}
by using $1-a_{ij}=b_{ij}$
and the identity $\sum_{t=0}^n (-1)^t q_i^{t(n-1)}
\pi_i^{\binom{t}{2}}\bbinom{n}{t}_i=0$.

Finally, if $k=i$ then we have

\begin{equation*}
C_{ij}^i
=\sum_{t=0}^{b_{ij}-1} (-1)^t q_i^{-t} e_i^{b_{ij}-t-1}e_je_i^{t}
\pi_i^{\binom{t}{2}+(t+1)p(j)}\parens{\pi_i^{a_{ij}}\bbinom{b_{ij}}{t}_i[b_{ij}-t]_i- \bbinom{b_{ij}}{t+1}_i[t+1]_i}=0,
\end{equation*}
by noting $\bbinom{n}{t}[n-t]=\bbinom{n}{t+1}[t+1]$ and $a_{ij}p(i)\in 2\Z$; see Condition~(d) in \S \ref{subset:covering}.
The lemma is proved.
\end{proof}

\begin{rmk}
A multi-parameter version of the quantum boson algebra can also be found in \cite{KKO}.
\end{rmk}

\subsection{$\bB$-modules}

For $\Z_2$-homogeneous elements $x,y \in \UU$, we write the supercommutator as
$[x,y]:=xy -\pi^{p(x)p(y)} yx$.
For $i\in I$, define $E_i'$ and $E_i''$ in $\End(\UU^-)$
by
\begin{equation}  \label{eq:E'}
[E_i,y]=\frac{\tJ_i\widetilde{K}_i E_i''(y)-\widetilde{K}_i^{-1} E_i'(y)}{\pi_i q_i-q_i^{-1}},   \qquad \text{ for } y\in \UU^-.
\end{equation}
The existence and uniqueness of such linear operators $E_i'$ and $E_i''$ is proved easily
(actually it can be read off from the proof Lemma~\ref{lem:E'Fcomms} below).

\begin{lem}\label{lem:E'Fcomms}
For $y\in \UU^-$, we have
\begin{align*}
E_i'(F_jy)&=\pi_i^{p(j)}q_i^{-\ang{\alpha_i^\vee, \alpha_j}}F_jE_i'(y)+\delta_{ij}y,\\
E_i''(F_jy)&=\pi_i^{p(j)} q_i^{\ang{\alpha_i^\vee, \alpha_j}}F_jE_i''(y)+\delta_{ij}y.
\end{align*}
\end{lem}

\begin{proof}
First we have
$$
[E_i,F_j y]=\pi_i^{p(j)} F_j [E_i,y]+ \delta_{ij} \frac{\tJ_i\tK_i- \tK_i^{-1}}{\pi_iq_i-q_i^{-1}} y.
$$
It follows by definition that
\begin{align*}
[E_i,F_j y] &=\frac{\tJ_i\widetilde{K}_i E_i''(F_j y)-\widetilde{K}_i^{-1} E_i'(F_j y)}
{\pi_i q_i-q_i^{-1}},
 \\
\pi_i^{p(j)} F_j [E_i,y] & + \delta_{ij} \frac{\tJ_i\tK_i- \tK_i^{-1}}{\pi_iq_i-q_i^{-1}} y
 \\
 &=  \frac{\pi_i^{p(j)} F_j\big(\tJ_i\widetilde{K}_i E_i''(y)-\widetilde{K}_i^{-1} E_i'(y)\big)}{\pi_i q_i-q_i^{-1}}
+ \delta_{ij} \frac{\tJ_i\tK_iy - \tK_i^{-1}y}{\pi_iq_i-q_i^{-1}}.
\end{align*}
Now the lemma follows by a comparison of the right-hand sides of the above two equations using the commutation relation \eqref{eq:relKF+JF}
and noting $\tJ_i$ is central.
\end{proof}

\begin{prop}  \label{prop:qcommute}
\label{Ka:3.4.5}
We have $E_i'E_j''=\pi^{p(i)p(j)}q_j^{\ang{\alpha_j^\vee, \alpha_i}}E_j''E_i'$, for $i,j\in I$.
\end{prop}

\begin{proof}
Let $\nu\in Q^-$ and recall that for $\mu=\sum_i a_i\alpha_i\in Q$
we define its height $\height \mu=\sum_i a_i$.
Let $y\in U^-_\nu$.

If $\height (-\nu) \leq 1$, then $E_i'E_j''(y)=0=E_j''E_i'(y)$.
Otherwise, we may assume $y=F_ky'$ for some $k\in I$
and $y'\in U^-$ with $\height y' < \height y$.
Then
\begin{align*}
E_i'E_j''(y)
&=E_i'(\pi_j^{p(k)}  q_j^{\ang{\alpha_j^\vee, \alpha_k}}F_kE_j''(y')+\delta_{jk}y')\\
&=f(i,j,k)F_kE_i'E_j''(y')+\pi_j^{p(k)} q_j^{\ang{\alpha_j^\vee, \alpha_k}}\delta_{ik}E_j''(y')+\delta_{jk}E_i'(y')\\
&=f(i,j,k)F_kE_i'E_j''(y')+\pi^{p(i)p(j)}q^{d_j\ang{\alpha_j^\vee, \alpha_i}}\delta_{ik}E_j''(y')+\delta_{jk}E_i'(y')
\end{align*}
and similarly
\begin{align*}
E_j''E_i'(y)&=f(i,j,k)F_kE_j''E_i'(y')+\pi^{p(i)p(j)}q^{-d_i\ang{\alpha_i^\vee, \alpha_j}}\delta_{jk}E_i'(y')+\delta_{ik}E_j''(y')
\end{align*}
where we have denoted $f(i,j,k)=\pi^{(\parity i+\parity j)\parity k} q_j^{\ang{\alpha_j^\vee, \alpha_k}}
q^{-d_i\ang{\alpha_i^\vee, \alpha_k}+d_j\ang{\alpha_j^\vee, \alpha_k}}$.

Note that $d_j\langle \alpha_j^\vee, \alpha_i\rangle =d_i\ang{\alpha_i^\vee, \alpha_j}$ and by induction $E_i'E_j''(y')
=\pi^{p(i)p(j)}q_j^{\ang{\alpha_j^\vee, \alpha_i}}E_j''E_i'(y')$. Therefore,
we have
$E_i'E_j''(y)=\pi^{p(i)p(j)}q_j^{\ang{\alpha_j^\vee, \alpha_i}}E_j''E_i'(y)$.
\end{proof}

From this we derive the following (see \cite[Lemma 1.3.15]{CHW} for an equivalent version,
and the equivalence can be read off from \eqref{eq:E=r} below).

\begin{cor}\label{cor:derivtozero}
For $x\in\UU^-$, if $E_i'(x)=0$ for all $i\in I$ then $x\in\Qqp$.
\end{cor}
\begin{proof}
The proof  proceeds as in \cite[Lemma 3.4.7]{Ka}.
\end{proof}

\begin{lem}\label{Ka:3.4.6}
Let $i\in I$ and $u\in \UU^-_{\zeta}$ such that $E_i'(u)=0$.
Then for any $\UU$-module $M$ and $m\in M_\lambda$ such that $e_i m=0$,
we have
\[\tK_i^n E_i^{n}um=\pi_i^{n\ang{\alpha_i^\vee,\lambda}}\frac{q_i^{n(2\ang{\alpha_i^\vee,
\lambda+\zeta}+3n+1)}}{(\pi_iq_i-q_i^{-1})^n} (E_i''^n u)m.
\]
\end{lem}

\begin{proof}
This lemma has essentially the same proof as \cite[Lemma 3.4.6]{Ka}.
The power of $\pi$ comes from the central element $\tJ_i$.
\end{proof}

Our interest in the superalgebra $\bB$ comes from the following result.

\begin{prop}\label{prop:bPaction}
$\UU^-$ is a $\bB'$-module as well as a $\bB$-module, where $f_i$ acts as multiplication by $F_i$
and $e_i$ acts by the map $E_i'$ for all $i\in I$.
\end{prop}

\begin{proof}
By Lemma \ref{lem:E'Fcomms}, $\UU^-$ is a $\bB'$-module.
Recall the Serre elements $S_{ij} \in \bB'$ from \eqref{eq:Sij} and denote the $f$-counterparts
by
\[
S'_{ij}=\sum_{t=0}^{b_{ij}} (-1)^t \pi_i^{\binom{t}{2}+tp(j)}
\bbinom{b_{ij}}{t} f_i^{b_{ij}-t}f_jf_i^{t} \in \bB'.
\]
To show that $\UU^-$ is a $\bB$-module, it suffices to show that
$S_{ij}$ and $S'_{ij}$ act as zero on any $y\in\UU^-$.
By the definition of the action and the Serre relations in $\UU$,
\[S'_{ij}y=\Big(\sum_{t=0}^{b_{ij}} (-1)^t \pi_i^{\binom{t}{2}+tp(j)}
\bbinom{b_{ij}}{t} F_i^{b_{ij}-t}F_jF_i^{t} \Big)y=0.\]

For $S_{ij}$, we may assume that $y$ is a monomial in the generators
$F_i$ for $i\in I$, so $y=m(f)1$ where $m(f)\in \bB'$ is a monomial
in the $f_i$ for $i\in I$. By repeated application of Lemma \ref{particleserre},
$S_{i,j}m(f)= c \ m(f)S_{ij}$ for some scalar $c\in \Qqp$.
Since $E_k'(1)=0$ for all $k\in I$, $S_{ij}1=0$
whence $S_{ij}y=0$.
\end{proof}

\begin{cor}
As $\bB$-modules, $\UU^-\cong\bB/\sum_i \bB e_i$.
\end{cor}

\begin{proof}
The final remark in the proof of Proposition \ref{prop:bPaction} shows that
there is a $\bB$-module homomorphism $\bB/\sum_i \bB e_i\rightarrow \UU^-$.
On the other hand, the $f_i$ generate a subalgebra of $\bB$ isomorphic
to $\UU^-$, so this map must be an isomorphism.
\end{proof}

\subsection{Polarization on $\UU^-$}
\label{subsec:bilformU-}

\begin{prop}\label{prop:Ubilprod}
There is a unique bilinear form  $(\cdot,\cdot)$ on $\UU^-$ satisfying
\[ (1,1)=1,\qquad (f_i y,z)=(y,e_iz)\quad \forall y,z\in\UU^-,\ i\in I.\]
Moreover, this bilinear form is symmetric.
\end{prop}

\begin{proof}
First note that there is a unique linear map $a:\bB\rightarrow \bB$
with $a(e_i)=f_i$ and $a(f_i)=e_i$, and $a(xy)=a(y)a(x)$ for $x,y\in \bB$.
Using this, $(\UU^-)^*$ becomes a $\bB$-module via
$(p\cdot\phi)(y)=\phi(a(p)\cdot y)$ for $p\in \bB$, $y\in \UU^-$ and
$\phi\in (\UU^-)^*$.

Let $\phi_0\in (\UU^-)^*$ be defined by $\phi_0(1)=1$ and
$\phi_0(\sum_i f_i\UU^-)=0$. Note that $e_i\phi_0(x)=\phi_0(f_ix)=0$
for all $x\in \UU^-$, $i\in I$. Therefore, there is a $\bB$-homomorphism
$\Psi:\UU^-\rightarrow (\UU^-)^*$ factoring through the map
$\bB/\sum\bB e_i\rightarrow (\UU^-)^*$; in particular, $1\mapsto \phi_0$.

Define $(\cdot,\cdot)$ on $\UU^-$ by $(y,z)=\Psi(y)(z)$. Then by construction,
$(1,1)=\phi_0(1)=1$ and $(f_iy,z)=f_i\Psi(y)(z)=\Psi(y)(e_iz)=(y,e_iz)$.
Clearly, these properties completely determine the bilinear form.
Then since the form $(\cdot,\cdot)'$ defined by $(y,z)'=\Psi(z)(y)$ satisfies the
same properties, the symmetry follows by the uniqueness of such a bilinear form.
\end{proof}

\begin{cor}
The bilinear form $(\cdot, \cdot)$ on $\UU^-$ is nondegenerate; moreover,
$(\UU^-_\nu,\UU^-_\mu)=0$ if $\nu\neq\mu$.
\end{cor}

\begin{proof}
The weight claim follows from the definition of the bilinear form and may be
shown by induction on the height of weights.
Nondegeneracy of the bilinear form may be shown also by induction on height with a crucial
observation as follows: if $0 \neq y\in\UU^-_\nu$ with $\nu \neq 0$ such that $(y,\UU^-_\nu)=0$,
then $(e_iy,\UU^-_{\nu+i})=0$ for all $i\in I$, whence $e_iy=0$ for all $i\in I$.
But then by Corollary \ref{cor:derivtozero} $\nu=0$, and hence we have a contradiction.
\end{proof}

Note that $\UU^- \cong \UU^-_+ \oplus \UU^-_-$  as $\bB$-modules, where $\UU^-_\pm :=\UU^-/\ang{\pi\pm 1}$.
The bilinear form $(\cdot,\cdot)$ on $\UU^-$ restricts to bilinear forms on $\UU^-_\pm$, still
denoted by $(\cdot,\cdot)$.
The bilinear form $(\cdot,\cdot)$ will be referred to as the {\em polarization} on $\UU^-$, $\UU^-_+$ or $\UU^-_-$.
Corollary~\ref{cor:derivtozero} implies the following.

\begin{cor}
The $\bB$-modules $\UU^-_\diamond$, for $\diamond \in \{+,-\}$, are simple.
\end{cor}

\subsection{Crystal basis of $\UU^-$}
\label{subsec:CBUU-}

We define a category $\curlyP$ as follows.
The objects of $\curlyP$ are $\bB$-modules $M$
such that for any $m\in M$, there exists an $t\in\N$ such that
for any $i_1,\ldots, i_t\in I$, $e_{i_1}\ldots e_{i_t}m=0$.
The homomorphisms are $\bB$-module homomorphisms.
Then we have $\UU^-_\diamond \in \curlyP$, for $\diamond \in \{\emptyset,+,-\}$,
where by convention we drop the subscript $\emptyset$ in case of $\diamond =\emptyset$.
In fact, $\UU^-_\diamond$, for $\diamond \in \{+,-\}$,  are the
only simple modules up to isomorphism and $\curlyP$ is semisimple.

\begin{lem}
Let $M\in \curlyP$. For each $i\in I$, every $m\in M$
has a unique expression of the form
\[m=\sum_{t\ge 0} f_i^{(t)}m_t\]
where $m_t\in \ker e_i$  and $m_t$ is nonzero for finitely many $t$.
\end{lem}

\begin{proof}
By the definition of $\curlyP$, each $e_i$ is locally finite on $M$.
Note that $e_i$ and $f_i$ generate a subalgebra of $\bB$
isomorphic to $\bF_{q_i, \pi_i}$ and so Lemma \ref{lem:Fprops}(ii)
finishes the proof.
\end{proof}

Let $i\in I$.
 Let $M\in \curlyP$ and $m\in M$
such that $m=\sum_{t} f_i^{(t)}m_t$ with $m_t\in \ker e_i$.
We define the Kashiwara operators
\[\tilde{e}_i m=\sum_{t} f_i^{(t-1)}m_t,
\quad \text{ and } \tilde{f}_i m=\sum_{t} f_i^{(t+1)}m_t.\]
Note that these operators (super)commute with $\bB$-module homomorphisms.

The action of $\bB$ on $\UU^-_+ =\UU^-/\ang{\pi-1}$ factors through $\bB/\ang{\pi-1}$, and then
we are in Kashiwara's original setting of $q$-boson algebra and its simple module. In this case,
It is well known (\cite{Ka}) that crystal basis on $\UU^-_+$ exists.
We shall formulate variants of the notion of crystal bases applicable to $\UU^-_\diamond$
where $\diamond \in \{\emptyset, -\}$. To that end, we consider the subcategory $\curlyP_-$ of $\curlyP$ which consists of $\bB$-modules
on which $\pi$ acts as $-1$,
as well as the subcategory $\curlyP_\pi$ of  $\curlyP$ which consists of $\bB$-modules which are free $\Qqp$-modules.

Let $\A\subseteq \Qq$ be the subring of functions regular at $q=0$.
Let $R$ be an arbitrary subring of $\Q$ with $1$.

\begin{dfn}\label{dfn:pibasis}
\begin{enumerate}
\item
For a free  $R^\pi$-module $F$, a {\bf $\pi$-basis} for $F$
is a subset $B$ of $F$ such that
$B^0 \subseteq B \subseteq \pi B^0\cup B^0$, for some basis $B^0$ of the free $R^\pi$-module $F$.

\item
For a free $R$-module $F$ on which $\pi$ acts as $-1$, a {\bf $\pi$-basis} for $F$
is a subset $B$ of $F$ such that
$B^0 \subseteq B \subseteq \pi B^0\cup B^0$, for some basis $B^0$ of the free $R$-module $F$.

\noindent (For $F$ on which $\pi$ acts as $1$, the definition forces a $\pi$-basis to be a genuine basis.)
%

\item Assume in addition that $F$ admits a non-degenerate  bilinear form $(\cdot,\cdot)$. The $\pi$-basis $B$ in (1) or (2) above is called
{\bf $\pi$-orthonormal} if $B^0$ is orthogonal with respect to $(\cdot,\cdot)$ and $(b,b)\in\{1,\pi\}$ for $b\in B^0$.
\end{enumerate}
Of course, a $\pi$-basis $B$ above gives rise to a ``maximal" $\pi$-basis $\pi B^0\cup B^0$.
\end{dfn}

\begin{dfn}  \label{dfn:crb}
A free $\A$-submodule $L$
of a $\bB$-module $M$ in the category $\curlyP_-$ is called a {\bf crystal lattice} if
\begin{enumerate}
\item $L\otimes_{\A} \Qq=M$;
\item $\te_i L\subseteq L$ and $\tf_i L \subseteq L$.
\end{enumerate}
(Note that $L/qL$ is a $\Q$-module.)
$M$ is said to have a {\bf crystal basis} $(L,B)$ if
a subset $B$ of $L/qL$ satisfies
\begin{enumerate}
\item[(3)] $B$ is a $\pi$-basis of $L/qL$; 
\item[(4)] $\te_i B\subseteq B\cup\set{0}$ and $\tf_i B\subseteq B$;
\item[(5)] For $b\in B$, if $\te_ib \neq 0$ then $b=\tf_i\te_ib$.
\end{enumerate}
A crystal basis $(L,B)$ is called {\bf maximal} if $B$ is a maximal $\pi$-basis of $L/qL$.
\end{dfn}

\begin{rmk}\label{rmk:piCrB}
Accordingly, {\em a {\bf crystal $\pi$-lattice} $L$ and a {\bf crystal $\pi$-basis} $(L,B)$ of a $\bB$-module $M$  in the category   $\curlyP_\pi$}
consists of a free $\A^\pi$-submodule $L$
of $M$ and a subset $B$ of the $\Q^\pi$-module $L/qL$
satisfying the axioms as in Definition~\ref{dfn:crb}
with (1) 
modified as
\begin{enumerate}
\item[($1'$)]
 $L\otimes_{\A^\pi} \Qqp=M$.
\end{enumerate}
Note the meaning of (3) is adjusted according to Definition~\ref{dfn:pibasis}.
\end{rmk}


We let $\cL(\infty)_+$ and $\cL(\infty)_-$ (reps. $\cL(\infty)$) be the $\A$-submodules of $\UU^-_+$ and $\UU^-_-$
(respectively, the $\Ap$-submodule of $\UU^-$)
generated by all possible $\tf_{i_1}\ldots\tf_{i_t} 1$. We let $B(\infty)_\diamond=\set{\tf_{i_1}\ldots\tf_{i_t}1}$ be
the subset of $\cL(\infty)_\diamond/q\cL(\infty)_\diamond$, where $\diamond \in \{\emptyset, +,-\}$ and
by convention again we drop the subscript $\diamond$ in case of $\diamond =\emptyset$. We shall prove in Section~ \ref{sec:grandloop} that
$(\cL(\infty)_\diamond, B(\infty)_\diamond)$ is a crystal basis of $\UU^-_\diamond$, for $\diamond \in \{\emptyset, +,-\}$.
Note that the bilinear form allows us to define a dual lattice in $\UU^-_\diamond$, for $\diamond \in \{\emptyset, +,-\}$ as follows:
\[\cL(\infty)_\pm^\vee=\set{u\in \UU^-_\pm \mid  (u,\cL(\infty)_\pm)\subseteq \A},
\qquad
\cL(\infty)^\vee=\set{u\in \UU^-\mid  (u,\cL(\infty))\subseteq \A^\pi}.\]


\section{Crystal bases and polarization} 
\label{sec:CBpolar}

In this section we formulate the main theorems of crystal bases for $\UU^-$ and integrable $\UU$-modules.
We also formulate the tensor product rule for crystal bases.

\subsection{Kashiwara operators for $\UU$-modules}

We start with the following observation.

\begin{lem}
Let $M\in \curlyO$. For each $i\in I$, every $m\in M^\lambda$ has a unique expression of the form
\[m=\sum_{t \ge 0} F_i^{(t)}m_t\]
where $m_t\in M^{\lambda+t\alpha_i}\cap \ker E_i$  are
nonzero for finitely many $t$.
\end{lem}

\begin{proof}
When $p(i)=1$, $M$ is a direct sum of simple
$\UU_i$-modules, where $\UU_i$ is the quantum group of $\osp(1|2)$; see
\cite{CW, CHW}.
Uniqueness is proved similarly to the case when $p(i)=0$ (see \cite[\S 2.2]{Ka}).
\end{proof}

\begin{dfn}
Let $m\in M^\lambda$ with
\[m=\sum_{t \ge 0} F_i^{(t)}m_t\]
where $m_t\in M^{\lambda+t\alpha_i}\cap \ker e_i$ are nonzero for finitely many $t$.
We define
\[\tilde{e}_i m=\sum_{t} F_i^{(t-1)}m_t, \qquad \tilde{f}_i m=\sum_{t} F_i^{(t+1)}m_t.\]
\end{dfn}
Note that $\te_i m\in M^{\lambda+\alpha_i}$ and  $\tf_i m\in M^{\lambda-\alpha_i}$.
Moreover, $\te_i$ and $\tf_i$ (super)commute with $\UU$-module homomorphisms.

Now recall the definition of the rings $\A$ and $\Ap$ from \S \ref{subsec:CBUU-}.

\begin{dfn}
Let $M$ be a $\UU$-module in the category $\curlyO$.
A free $\A$-submodule $\cL$
of $M$ is called a crystal lattice of $M$ if
\begin{enumerate}
\item $\cL\otimes_\A \Qq=M$;
\item $\cL=\bigoplus_{\lambda\in P} \cL_\lambda$ where
$\cL_\lambda=\cL\cap M_\lambda$ for all $\lambda\in P$;
\item $\te_i \cL\subseteq \cL$ and $\tf_i \cL \subseteq \cL$.
\end{enumerate}
A pair $(\cL,B)$ is called a crystal basis of $M$ if a subset $B$ of the $\Q$-module $\cL/q\cL$
satisfies
\begin{enumerate}
\item[(4)] $B$ is a $\pi$-basis of $\cL/q\cL$ over $\Q$;
\item[(5)] $B=\coprod_{\lambda\in P} B_\lambda$ where
$B_\lambda=B\cap \cL_\lambda/q\cL_\lambda$,
\item[(6)] $\te_i B\subseteq B\cup\set{0}$ and $\tf_i B\subseteq B\cup\set{0}$;
\item[(7)] For $b,b'\in B$, $\te_ib = b'$ if and only if $b=\tf_ib'$.
\end{enumerate}
\end{dfn}
Also a $\pi$-version of crystal basis for $\Qqp$-free integrable modules can be formulated
similarly as in Remark~\ref{rmk:piCrB}.

\begin{rmk}\label{rmk:uniqueness}
We shall set out to prove the existence of the crystal bases for the integrable modules $V(\lambda)_\pm$, for $\la \in P^+$.
Assume for the moment that we have done this. Since these axioms are unaffected under direct
sums of lattices and parity changes, we can endow any integrable module
$M$ with a crystal basis built out of direct sums of the simples.
Uniqueness of a maximal crystal basis on $M$ (up to isomorphism) can be proved by the same arguments
as in \cite[\S 2.6]{Ka}.
\end{rmk}

\begin{example}
\label{ex:rank1}
Let $I=\set{i}$. Then the simple modules are $(n+1)$-dimensional modules
$V(n)_\pm$ for $n\in \Z_{\ge 0}$. Let  $v_n^+$ denote a highest weight vector in $V(n)_\pm$.  Define the $\A$-lattice
$\cL(n)_\pm =\bigoplus_{k=0}^n \A F^{(k)}v_n^+$ in $V(n)_\pm$, and
$B(n)_\pm =\set{F^{(k)}v_n^+ +q\cL(n)_\pm \mid 0\le k \le n}$ (the index $i$ is suppressed here).
Then $(\cL(n)_\pm , B(n)_\pm)$ is a crystal basis of $V(n)_\pm$. In this case, $B(n)_\pm$
is actually a genuine $\Q$-basis for $\cL(n)_\pm/q\cL(n)_\pm$.
\end{example}

\begin{example}
\label{ex:CBla}
Let $\lambda\in P^+$ and let $v^+_\lambda$ be a highest weight vector
of $V(\lambda)_\pm$. Consider the subset
$B(\lambda)_\pm :=\set{\tf_{i_1}\ldots \tf_{i_t} v^+_{\lambda}}\setminus\set 0$ of $V(\lambda)_\pm$.
Let $\cL(\lambda)_\pm$ be the $\A$-submodule of $V(\la)_\pm$ generated by $B(\lambda)_\pm$.
We shall prove in Section~ \ref{sec:grandloop}
that $(\cL(\lambda)_\pm, B(\lambda)_\pm)$ is a (minimal) crystal basis in contrast to the maximal
crystal basis $(\cL(\lambda)_\pm, \pi B(\lambda)_\pm \cup B(\lambda)_\pm)$ in the sense of Definition~\ref{dfn:pibasis}
(of course, the case of $+$ was already in \cite{Ka}).

\end{example}

\begin{example}
\label{ex:oddodd}
Assume that $I_1$ contains $i,j$ such that $a_{ij}=a_{ji}=0$.  Then $\tf_i\tf_j=\pi \tf_j\tf_i$,
and hence $\pi B(\lambda)_- \cap B(\lambda)_- \not=\emptyset$ for various $\la \in P^+$.
\end{example}

\subsection{Polarization}
\label{subsec:polar}

Let $\tau_1:\UU\rightarrow \UU$ be the anti-automorphism defined by
\[\tau_1(E_i)=q_i^{-1}\tK_i^{-1}F_i,\,\,
\tau_1(F_i)=q_i^{-1}\tK_iE_i,\,\,
\tau_1(K_\mu)=K_\mu,\,\,  \tau_1(J_\mu)=J_\mu,
\quad (i\in I, \mu \in P^\vee)
\]
such that $\tau_1(xy)=\tau_1(y)\tau_1(x)$ for $x, y \in \UU$. One checks that $\tau_1^2=1$.
(Note a typo in \cite[\S2.5]{Ka} where $q_it_i^{-1}f_i$ should read $q_i^{-1}t_i^{-1}f_i$.)

\begin{prop}\label{prop:pol}
Let $\la \in P^+$.
There is a unique bilinear form  $(\cdot,\cdot)$ on $V(\lambda)_+$ and $V(\la)_-$ respectively,
which satisfies $(v^+_\lambda,v^+_\lambda)=1$ and
\begin{equation}\label{eq:polar}
 (u v,w)
=(v,\tau_1(u)w), \qquad \forall u\in\UU,\ v,w\in V(\lambda)_\pm.\end{equation}
Moreover, this bilinear form on $V(\la)_\pm$ is symmetric.
\end{prop}

Recall the $\A$-lattices $\cL(\la)_\pm$ of $V(\la)_\pm$ from Example~\ref{ex:CBla}.
We define the dual lattices in $V(\la)_\pm$ to be
\[\cL(\lambda)_\pm^\vee=\set{v\in V(\lambda)_\pm \mid (v,\cL(\lambda)_\pm)\subseteq \A}.\]

For a weight $\UU$-module $M$, we call a bilinear form $(\cdot,\cdot)$
on $M$ a {\em polarization} if (\ref{eq:polar}) is satisfied with $M$ in place of $V(\la)_\pm$. Note that if $m\in M^\lambda$
and $m'\in M^\mu$, then
\begin{equation}\label{eqn:polweight}
(m,m')=0\text{ unless }\lambda= \mu\text{ and }p(m)= p(m').
\end{equation}

Recall \cite[Lemma~2.5.1]{Ka} that  the tensor product of modules admitting polarizations
also admits a natural polarization given by the tensor of the bilinear forms. In our super setting, this is not quite true due to the additional asymmetry
in the definition of the coproduct (as noted in Remark \ref{rem:coproducts}).
Recall $\catO^{\rm int}_\pm$ from \S \ref{subsec:cat}.

\begin{lem} \label{lem:pseudopol}
Assume  that either  $(1)$ $M, N\in \catOint$ are free $\Qqp$-modules,
or $(2)$   $M, N\in \catO^{\rm int}_+$, or $(3)$   $M, N\in \catO^{\rm int}_-$.
Assume $M$ and $N$ admit  polarizations $(\cdot,\cdot)$.
Then the symmetric bilinear form on the module $M\otimes N$ given by
$(m_1\otimes n_1, m_2\otimes n_2) =(m_1,m_2)(n_1,n_2)$ satisfies
$$
\big(\Delta(u)(m_1\otimes m_2), n_1 \otimes n_2 \big)=\big(m_1\otimes m_2,\Delta'(\tau_1(u))(n_1 \otimes n_2)\big),
$$
for $u\in \UU, m_1, m_2 \in M, n_1, n_2 \in N.$
\end{lem}
We call such a bilinear form on $M\otimes N$ a {\em $J$-polarization}, as the difference on $\Delta$ and $\Delta'$
is caused by the $J_\mu$'s.

\begin{proof}
Let $m_1\in M_\mu$ and $m_2\in M_{\mu'}$.
By a direct computation, we have
\begin{align}  \label{eq:DF}
\big(\Delta(F_i) (m_1\otimes n_1), m_2\otimes n_2 \big)   &
=(F_i m_1,m_2)(n_1, n_2)+\pi_i^{p(m_1)}q^{\ang{\alpha_i^\vee,\mu}}(m_1,m_2)(F_i n_1, n_2),
\end{align}
and
\begin{align}  \label{eq:DF'}
\big(m_1\otimes n_1&, \Delta'(q_i^{-1}K_iE_i) (m_2\otimes n_2) \big)  \notag \\
&=(m_1, q_i^{-1}K_iE_im_2)(n_1, n_2)
+\pi_i^{p(m_2)}q_i^{\ang{\alpha_i^\vee,\mu'}}(m_1,m_2)(n_1,q_i^{-1}K_iE_in_2).
\end{align}
By (\ref{eqn:polweight}) and Proposition~ \ref{prop:pol}, \eqref{eq:DF} and \eqref{eq:DF'} are equal.
\end{proof}

\begin{rmk}
To see why we need $\Delta'$ in Lemma~\ref{lem:pseudopol}, we compute using $\Delta$ in replace of $\Delta'$ that
\begin{align}
\label{eq:DE}
\big(m_1\otimes n_1&, \Delta(\tau_1(E_i)) (m_2\otimes n_2) \big)  \notag \\
&=(m_1, \tau_1(E_i)m_2)(n_1, n_2)
+\pi_i^{p(m_2)}(\pi_iq_i)^{\ang{\alpha_i^\vee,\mu'}}(m_1,m_2)(n_1,\tau_1(E_i)n_2).
\end{align}
In particular, if $\parity i=1$ and $\ang{\alpha_i^\vee,\mu'}\not\in 2\Z$,
then \eqref{eq:DE} is not equal to \eqref{eq:DF}.
\end{rmk}

For $\lambda, \mu\in P^+$ and $\diamond \in \{+,-\}$, we define the even $\UU$-module homomorphisms
\begin{align}
\begin{split}
 \Phi(\lambda,\mu):V(\lambda+\mu)_\diamond \rightarrow V(\lambda)_\diamond\otimes V(\mu)_\diamond,
&\qquad  v^+_{\lambda+\mu}\mapsto v^+_\lambda\otimes v^+_\mu, \\
 \Phi'(\lambda,\mu):V(\lambda+\mu)_\diamond\rightarrow V(\lambda)_\diamond\otimes' V(\mu)_\diamond,
&\qquad v^+_{\lambda+\mu}\mapsto v^+_\lambda\otimes v^+_\mu, \\
 \Psi(\lambda,\mu):V(\lambda)_\diamond\otimes V(\mu)_\diamond\rightarrow V(\lambda+\mu)_\diamond,
&\qquad v^+_\lambda\otimes v^+_\mu\mapsto v^+_{\lambda+\mu}, \\
 \Psi'(\lambda,\mu):V(\lambda)_\diamond\otimes' V(\mu)_\diamond\rightarrow V(\lambda+\mu)_\diamond,
&\qquad v^+_\lambda\otimes v^+_\mu\mapsto v^+_{\lambda+\mu}.
\label{eqn:tensorembeddings}
\end{split}
\end{align}

Then $\Psi(\lambda,\mu)\circ\Phi(\lambda,\mu)$ and $\Psi'(\lambda,\mu)\circ\Phi'(\lambda,\mu)$
are the identity map on $V(\lambda+\mu)_\diamond$. In particular,
$V(\lambda)_\diamond\otimes V(\mu)_\diamond=\im\Phi\oplus\ker\Psi$ and
$V(\lambda)_\diamond\otimes' V(\mu)_\diamond=\im\Phi'\oplus\ker\Psi'$.
Note that these maps, being even $\UU$-module homomorphisms, commute with
Kashiwara operators.

\begin{cor}
Let $\lambda, \mu\in P^+$ and $\diamond \in \{+,-\}$.
Let $(\cdot, \cdot)$ denote the polarization on $V(\lambda+\mu)_\diamond$
and the $J$-polarization on $V(\lambda)_\diamond\otimes V(\mu)_\diamond$.
Then we have
$$
(\Psi(\lambda,\mu)(w),v)=(w,\Phi'(\lambda,\mu)(v)), \qquad
(\Psi'(\lambda,\mu)(w),v)=(w,\Phi(\lambda,\mu)(v)),
$$
for $v\in V(\lambda+\mu)_\diamond$ and $w\in V(\lambda)_\diamond \otimes V(\mu)_\diamond$.
\end{cor}

\subsection{Odd rank 1 calculation}
\label{subsec:rank1}

Let $n \in \Z_{\ge 0}$ and $I=\set{i}$ with $p(i)=1$.
We consider the module $V(n) \otimes V(1)$.
This module has two submodules over $\Qqp$ generated by singular vectors:
a submodule $N_1$ generated by the (even) singular vector
\[w=v^+_n \otimes v^+_1\]
and $N_2$ generated by the (odd) singular vector
\[z=v^+_n\otimes Fv^+_1- \pi^n q [n]^{-1} Fv^+_n\otimes v^+_1.\]
We directly compute
\begin{align*}
F^{(k)}w&=F^{(k)}v^+_n\otimes v^+_1 + \pi^n(\pi q)^{n+1-k}F^{(k-1)}v^+_n\otimes Fv^+_1,
 \\
F^{(k)}z &=(1- \pi(\pi q)^{n-k}[n]^{-1}[k]) F^{(k)} v^+_n\otimes Fv^+_1
-\pi^n q [n]^{-1}[k+1]F^{(k+1)}v^+_n\otimes v^+_1.
\end{align*}

Observing that $\pi(\pi q)^{n-k}[n]^{-1}[k]\in q^{2n-2k}\A^\pi$, we have
\begin{equation}
F^{(k)}w=
\begin{cases}
F^{(k)}v^+_n\otimes v^+_1 & \text{ if } 0\leq k<n+1\\
\pi^n F^{(n)}v^+_n\otimes Fv^+_1 & \text{ if } k=n+1
\end{cases}\quad \mathrm{ mod }\ q\cL,
\end{equation}
and
\begin{equation}
F^{(k)}z=F^{(k)} v^+_n\otimes Fv^+_1\quad \mathrm{ mod }\ q\cL,\quad 0\leq k \leq n-1.
\end{equation}

In particular, $V(n)\otimes V(1) \cong N_1\oplus N_2\cong V(n+1)\oplus V(n-1)$.

The above calculations remain to make perfect sense for $V(n)_\pm\otimes V(1)_\pm$, with $\pi$ in
the formulas above interpreted as $1$ and $-1$ accordingly.
In particular,
$$V(n)_\pm\otimes V(1)_\pm\cong N_1\oplus N_2\cong V(n+1)_\pm\oplus V(n-1)_\pm.
$$

\subsection{Tensor product rule for crystal bases}

We can use the calculations in \S \ref{subsec:rank1}  to prove a tensor product rule for crystal bases in general.
Let $M$ be an integrable $\UU$-module with crystal basis $(\cL,B)$.
For each $i\in I$ and $b\in B$, define
\begin{align}  \label{dfn:stringindex}
\begin{split}
\varphi_i(b) &=\max\set{n \mid \tf_i^nb\neq 0},
  \\
\varepsilon_i(b) &=\max\set{n \mid \te_i^nb\neq 0}.
\end{split}
\end{align}
We note that $\varphi_i(b)=\ang{\alpha_i^\vee, \mu}+\varepsilon_i(b)$ for $b \in B_\mu$.

\begin{thm}\label{thm:tensorproductrule}
Let $M,M'\in\catO^{\rm int}_-$ (reps. $M,M'\in\catO^{\rm int}_+$) be modules with crystal bases
$(\cL,B)$ and $(\cL',B')$.
Let $B\otimes B'=\set{b\otimes b'\in (\cL/q\cL)\otimes_{\Q} (\cL'/q\cL'):b\in B, b'\in B'}$.
The tensor product $M\otimes M'$ has a crystal basis $(\cL\otimes \cL', B\otimes B')$
subject to the rules:
\[\tf_i(b\otimes b')=\begin{cases}
\tf_i b\otimes b'
&\text{ if } \varphi_i(b)>\varepsilon_i(b'), \\
\pi_i^{p(b)} b\otimes \tf_i b' &\text{ otherwise};\end{cases}\]
\[\te_i(b\otimes b')=\begin{cases}
\pi_i^{p(b)} b\otimes\te_i b'
&\text{ if } \varphi_i(b)<\varepsilon_i(b'), \\
\te_i b\otimes  b' &\text{ otherwise.} \end{cases}\]
(All equalities are understood in $\cL\otimes \cL'/q\cL\otimes \cL'$.)
\end{thm}

\begin{proof}
It is sufficient to prove this for a fixed $i$, in which case
the theorem reduces to a statement for $I=\set{i}$. When $p(i)=0$,
the theorem is \cite[Theorem 1]{Ka}. Assume $p(i)=1$.
The case for $M,M'\in\catO^{\rm int}_+$ is again reduced to Kashiwara's original setting,
so we assume $M,M'\in\catO^{\rm int}_-$.
Since modules are completely reducible \cite{CW}, it suffices to prove this for tensor products of
 simple modules  $V(n)_-\otimes V(m)_-$, by induction on $m$.
 Recall from Example~\ref{ex:rank1} that $(\cL(n)_\pm,B(n)_\pm)$ is a crystal basis for $V(n)_\pm$.

From the odd rank 1 calculation in \S \ref{subsec:rank1}, the theorem holds for $V(n)_-\otimes V(1)_-$. This takes care the base case of induction.

By induction, we assume the theorem holds for $V(n)_-\otimes V(m)_-$. Note that
\[
V(n)_-\otimes V(m)_-\otimes V(1)_-\cong V(n)_-\otimes (V(m+1)_-\oplus V(m-1)_-).
\]
By the complete reducibility and the base case proved above, we conclude that
$\big(\cL(n)_-\otimes \cL(m)_-\otimes \cL(1)_-, B(n)_-\otimes B(m)_-\otimes B(1)_-\big)$
is a crystal basis of $V(n)_-\otimes V(m)_-\otimes V(1)_-$.
Moreover, $\big(\cL(n)_-\otimes \cL(m)_-\otimes \cL(1)_-, B(n)_-\otimes B(m)_-\otimes B(1)_-\big)$
decomposes as $\big(\cL(n)_-\otimes (\cL(m+1)_-\oplus \cL(m-1)_-), B(n)_-\otimes (B(m+1)_-\cup B(m-1)_-) \big)$.
Therefore, $(\cL(n)_-\otimes \cL(m+1)_-, B(n)_-\otimes B(m+1)_-)$
is a crystal basis of $V(n)_-\otimes V(m+1)_-$.
\end{proof}

\begin{rmk}
Jeong \cite{Jeo} claimed  a version of tensor product rule
for $\osp(1|2)$ without super signs. The proof was much more complicated in the setting of \cite{Jeo},
as  $V(n)$ for $n$ odd (and in particular $V(1)$) were not available.
The formulation and the calculations there were incorrect since he missed
the super signs in the quantum integers and the super signs arising from the multiplication in the tensor algebra;
see (\ref{eqn:STM}). For example,
the correct calculation of \cite[Eq.~ (4.2)]{Jeo} should read (it is understood that $\pi=-1$ below)
\[a_m=\frac{1}{[m]}\sum_{k=0}^{m-1} (\pi q^{-2})^kq^n= \pi^{m-1}q^{n+1-m}.\]
In particular, $\tf^{n+1}(v_n^+\otimes v_2^+)\equiv
\pi^n\tf^nv_n^+\otimes \tf_i v_1^+ \text{ mod }qA$.

On the other hand, there are no signs in the tensor product
rule presented in \cite[Theorem 4.1]{MZ}, since the signs
are absorbed into the various factors of $\sqrt{-1}$ therein.
\end{rmk}

\subsection{Main results on crystal bases}
 \label{sec:thmcrystal}

We are now ready to formulate the main theorems on crystal bases for $\UU^-$ and integrable modules $V(\la)$.

\begin{thm}  \label{th:CrBinf}
For $\diamond \in \{\emptyset, +,-\}$, $(\cL(\infty)_\diamond, B(\infty)_\diamond)$ is a crystal basis of $\UU^-_\diamond$.
\end{thm}

\begin{thm}   \label{th:CrBla}
Let $\la \in P^+$. For $\diamond \in \{ +, -, \emptyset\}$,
$(\cL(\la)_\diamond, B(\la)_\diamond)$ is a crystal basis of $V(\la)_\diamond$.
\end{thm}

For $\la \in P^+$ and $\diamond \in \{\emptyset, +, -\}$,  we define the (even) $\UU^-$-linear
projection map
\begin{equation}
\wp_\lambda:\UU^-_\diamond \longrightarrow V(\lambda)_\diamond,\quad 1\mapsto v^+_\lambda.
\end{equation}

\begin{thm}   \label{th:CrBinfla}
Let $\la \in P^+$. For $\diamond \in \{+, -,  \emptyset\}$,
\begin{enumerate}
\item
$\wp_\la (\cL(\infty)_\diamond) =\cL(\la)_\diamond$;

(This induces a homomorphism
$\overline{\wp}_\la: \cL(\infty)_\diamond/q \cL(\infty)_\diamond \rightarrow \cL(\la)_\diamond/q \cL(\la)_\diamond$.)

\item
$\overline{\wp}_\la$ sends $\{b \in B(\infty)_\diamond \mid \overline{\wp}_\la(b)\neq0\}$ isomorphically to $B(\la)_\diamond$;

\item
if $b\in B(\infty)_\diamond$ satisfies $\overline{\wp}_\la(b)\neq 0$,
then $\te_i \overline{\wp}_\la(b)=\overline{\wp}_\la(\te_i b).$
\end{enumerate}
\end{thm}

The proofs of these three theorems on crystal bases will be given in the following section.


\section{The grand loop argument}
\label{sec:grandloop}

In this section, $\pi$ is understood as its specialization to $-1$, and {\bf we adopt the convention (in this section only) of
dropping the subscript $``-"$ everywhere} in
$V(\la)$, $\cL(\la)$, $B(\la),$ $\UU, \UU^-, \cL(\infty), B(\infty)$ for all $\la$  for notational simplicity.
 (If instead we specialize $\pi$ to $1$ and add subscript $``+"$ everywhere, we
 would be back to the setting of Kashiwara \cite[\S4]{Ka}).
We prove the existence of crystal bases
using a modified version of Kashiwara's
grand loop argument, and will present only the parts which differ most from \cite[\S4]{Ka}.

For $\lambda, \mu\in P^+$, recall the maps $\Psi(\lambda,\mu)$ and
$\Phi(\lambda,\mu)$ from (\ref{eqn:tensorembeddings}).
Note that these maps, being (even) $\UU$-module homomorphisms, commute with the
Kashiwara operators.
We also define a map
\begin{equation}S(\lambda,\mu):V(\lambda)\otimes V(\mu)
\longrightarrow V(\lambda)
\end{equation}
by $S(\lambda,\mu)(u\otimes v^+_\mu)=u$ and
$S(\lambda,\mu)(V(\lambda)\otimes \sum f_i V(\mu))=0$.
This is an (even) $\UU^-$-linear map. Therefore, we have an (even) $\UU^-$-linear
map $S(\lambda,\mu)\circ\Phi(\lambda,\mu):V(\lambda+\mu)\longrightarrow V(\lambda)$
sending $v^+_{\lambda+\mu}$ to $v^+_\lambda$.

For $\zeta\in Q^-$, set $$\height\ \zeta=\sum n_i\mbox{ if }\zeta=-\sum n_i  \alpha_i.$$
Then let $Q^-(l)=\set{\zeta\in Q^-: \height\ \zeta\leq l}$.

Let $C_l$ be the collection of the following statements.

\begin{enumerate}
\item[($C_l.1$)] For $\zeta\in Q^-(l)$,
                    \[\te_i \cL(\infty)^\zeta\subseteq \cL(\infty).\]

\item[($C_l.2$)] For $\zeta\in Q^-(l)$ and $\lambda\in P^+$,
                    \[\te_i \cL(\lambda)^{\lambda+\zeta}\subset \cL(\lambda).\]

\item[($C_l.3$)] For $\zeta\in Q^-(l)$ and $\lambda\in P^+$,
                    $\wp_\lambda \cL(\infty)^{\zeta}= \cL(\lambda)^{\lambda+\zeta}$.

\item[($C_l.4$)] For $\zeta\in Q^-(l)$,
                    $B(\infty)^\zeta$ is a $\pi$-basis of
                    $\cL(\infty)^\zeta/q\cL(\infty)^\zeta$.

\item[($C_l.5$)] For $\zeta\in Q^-(l)$ and $\lambda\in P^+$,
                    $B(\lambda)^{\lambda+\zeta}$ is a $\pi$-basis of
                    $\cL(\lambda)^{\lambda+\zeta}/q\cL(\lambda)^{\lambda+\zeta}$.

\item[($C_l.6$)] For $\zeta\in Q^-(l-1)$ and $\lambda\in P^+$,
                    $\tf_i(xv^+_\lambda)\equiv (\tf_i x)v^+_\lambda$
                    mod $q\cL(\lambda)$ for $x\in \cL(\infty)^\zeta$.

\item[($C_l.7$)] For $\zeta\in Q^-(l)$ and $\lambda\in P^+$,
                    $\te_i B(\infty)^{\zeta}\subset B(\infty)\cup \set{0}$
                    and
                    $\te_i B(\lambda)^{\lambda+\zeta}
                    \subset B(\lambda)\cup \set{0}$.

\item[($C_l.8$)] For $\zeta\in Q^-(l)$ and $\lambda,\mu\in P^+$,
                    \[\Phi(\lambda,\mu)(\cL(\lambda+\mu)^{\lambda+\mu+\zeta})
                    \subset \cL(\lambda)\otimes \cL(\mu).\]

\item[($C_l.9$)] For $\zeta\in Q^-(l)$ and $\lambda,\mu\in P^+$,
                    \[\Psi(\lambda,\mu)\parens{\parens{\cL(\lambda)
                    \otimes \cL(\mu)}^{\lambda+\mu+\zeta}}
                    \subset \cL(\lambda+\mu).\]

\item[($C_l.10$)] For $\zeta\in Q^-(l)$ and $\lambda,\mu\in P^+$,
                    \[\Psi(\lambda,\mu)\parens{\parens{B(\lambda)
                    \otimes B(\mu)}^{\lambda+\mu+\zeta}}
                    \subset B(\lambda+\mu)\cup\set{0}.\]

\item[($C_l.11$)] For $\zeta\in Q^-(l)$ and $\lambda\in P^+$,
                    \[\set{b\in B(\infty)^\zeta:\hat \wp_\lambda(b)\neq 0}
                    \rightarrow B(\lambda)^{\lambda+\zeta}\]
                    is a bijection, where $\hat \wp_\lambda$
                    is  the map induced by $\wp_\lambda$.

\item[($C_l.12$)] For $\zeta\in Q^-(l)$, $\lambda\in P^+$
                    and $b\in B(\infty)^\zeta$ such that
                    $\hat \wp_\lambda(b)\neq 0$, we have
                    \[\te_i\hat \wp_\lambda(b)=\hat \wp_\lambda(\te_i(b)).\]

\item[($C_l.13$)] For $\zeta\in Q^-(l)$, $\lambda\in P^+$
                    $b\in B(\lambda)^{\lambda+\zeta}$ and
                    $b'\in B(\lambda)^{\lambda+\zeta+\alpha_i}$,
                    $b=\tf_ib'$ if and only if $b'=\te_ib$.

\item[($C_l.14$)] For $\zeta\in Q^-(l)$ and $b\in B(\infty)$, if $\te_ib\neq 0$
                    then $b=\tf_i\te_ib$.

\end{enumerate}

In the case $I=I_0$, Kashiwara (cf. \cite{Ka}) proved these statements via an induction
on $l$. These arguments can be adapted to our super setting,
with the main change being book-keeping for the power of $\pi$;
we will formulate precisely the $\pi$-modified results with proofs.
It is worth noting that Lemmas \ref{Ka:4.7.1}-\ref{Ka:4.7.3} (corresponding
to \cite[Lemmas~ 4.7.1-4.7.3]{Ka}) deviate the most in the general case. This is a
consequence of the failure of the natural bilinear form on tensors to
necessarily be a polarization; see Lemma \ref{lem:pseudopol}.

\begin{lem}[{\cite[Lemma 4.3.1]{Ka}}]\label{Ka:4.3.1}
Let $\zeta\in Q^-(l-1)$, $\lambda\in P^+$ and $u\in \cL(\infty)^\zeta$
(resp. $u\in \cL(\lambda)^{\lambda+\zeta}$).
If $u=\sum F_i^{(n)}u_n$
and if $e_iu_n=0$ (resp. $u_n\in V(\lambda)^{\lambda+\zeta+n\alpha_i}$,
$E_iu_n=0$ and $u_n=0$ except when $\ang{\alpha_i^\vee, \lambda+\zeta+n\alpha_i}\geq n$)
then all $u_n$ belong to $\cL(\infty)$ (resp. $\cL(\lambda)$).

If moreover
$u\mod q\cL(\infty)$ (resp. $q\cL(\lambda$) belongs to $B(\infty)$
(resp. $B(\lambda)$), then there exists $n$ such that $u=f_i^{(n)}u_n \mod q\cL(\infty)$ (resp. $q\cL(\lambda)$).
\end{lem}

\begin{proof}
By ($C_{l-1}.1$), $\te_i^tu\in \cL(\infty)$ for all $t$.
Let $m$ be the largest integer such that $u_m\notin \cL(\infty)$.
Then $\te_i^mu=\sum_{n\geq 0} f_i^{(n-m)}u_{n}$. Since
$u_{n}\in \cL(\infty)$ for $n>m$, $u_m\in \cL(\infty)$, a contradiction.
Therefore, $u_n\in \cL(\infty)$ for all $n$. A similar proof applies to
$\cL(\lambda)$.

Now suppose $u+q\cL(\infty)\in B(\infty)$. Let $n$ be the largest integer
such that $u_n\notin q\cL(\infty)$.
Then $\te_i^nu+q\cL(\infty)= u_n+q\cL(\infty)\neq q\cL(\infty)$.
By ($C_k.14$) for $k\leq l-1$,
$u+q\cL(\infty)=\tf_i^n\te_i^nu+q\cL(\infty)=f_i^{(n)}u_n+q\cL(\infty)$.
\end{proof}

Recall the map $\varepsilon_i$ from \eqref{dfn:stringindex}.
By the previous lemma, for $\zeta\in Q^-(l-1)$ and $b\in B(\infty)^\zeta$
(resp. $b\in B(\lambda)^{\zeta+\lambda}$), there exists
$u\in \cL(\infty)^{\zeta+n\alpha_i}$ (resp. $u\in \cL(\lambda)^{\lambda+\zeta+n\alpha_i}$)
such that $e_i u=0$ and $b+q\cL(\infty)=f_i^{(\varepsilon_i(b))}u+q\cL(\infty)$
(resp. $b+q\cL(\lambda)=f_i^{(\varepsilon_i(b))}u+q\cL(\lambda)$);
moreover, $u+q\cL(\infty)\in B(\infty)$ (resp. $u+q\cL(\lambda)\in B(\lambda)$).

The following is a $\pi$-analogue of {\cite[Lemma 4.3.2]{Ka}}.
\begin{lem}  \label{Ka:4.3.2}
Let $\zeta, \zeta'\in Q^-(l-1)$, $\lambda, \mu\in P^+$ and $i\in I$.

\begin{enumerate}
\item[(i)] $\tf_i(\cL(\lambda)^{\lambda+\zeta}\otimes \cL(\mu)^{\mu+\zeta'})
                        \subset \cL(\lambda)\otimes \cL(\mu)$
            and $\te_i(\cL(\lambda)^{\lambda+\zeta}\otimes \cL(\mu)^{\mu+\zeta'})
                        \subset \cL(\lambda)\otimes \cL(\mu)$
\item[(ii)]If $b\in B(\lambda)^{\lambda+\zeta}$ and $b'\in B(\mu)^{\mu+\zeta'}$,
            then we have
            \[\tf_i(b\otimes b')=\begin{cases}
            \tf_i b\otimes b'
            &\text{ if } \ang{\alpha_i^\vee, \lambda+\zeta}+\varepsilon_i(b)>\varepsilon_i(b'), \\
             \pi_i^{p(b)} b\otimes \tf_i b' &\text{ otherwise};\end{cases}\]
            \[\te_i(b\otimes b')=\begin{cases}
            \pi_i^{p(b)} b\otimes\te_i b' &\text{ if } \ang{\alpha_i^\vee, \lambda+\zeta}+\varepsilon_i(b)<\varepsilon_i(b'), \\
            \te_i b\otimes  b' &\text{ otherwise,}
            \end{cases}\]
            where all equalities are in
            $\cL(\lambda)\otimes \cL(\mu)/q\cL(\lambda)\otimes \cL(\mu)$.
\item[(iii)] For $b\otimes b'\in B(\lambda)^{\lambda+\zeta}
                    \otimes B(\mu)^{\mu+\zeta'}$, $\te_i(b\otimes b')\neq 0$
                implies that $b\otimes b'=\tf_i\te_i(b\otimes b')$.
\item[(iv)] For $b\in B(\lambda)^{\lambda+\zeta}$ and $b'\in B(\mu)^{\mu+\zeta'}$,
            if $\te_i(b\otimes b')=0$ for any $i$, then $\zeta=0$
            and $b=v^+_\lambda+q\cL(\lambda)$ or $b=\pi v^+_\lambda+q\cL(\lambda)$.
\item[(v)]  For $b\in B(\lambda)^{\lambda+\zeta}$,
             $\tf_i(b\otimes v^+_\mu)=\tf_i b\otimes v^+_\mu$ or $\tf_ib=0$.
\end{enumerate}
\end{lem}

\begin{proof}
(i). By Lemma \ref{Ka:4.3.1}, it is enough to show that for
$u\in \cL(\lambda)^{\lambda+\zeta+ni'}$ and $u'\in \cL(\mu)^{\mu+\zeta'+m\alpha_i}$
such that $E_iu=0$, $E_iu'=0$, $\ang{\alpha_i^\vee,\lambda+\zeta+n\alpha_i}\geq n$
and $\ang{\alpha_i^\vee,\lambda+\zeta+m\alpha_i}\geq m$, then
\[\tf_i(F_i^{(n)}u\otimes F_i^{(m)}u')\in \cL(\lambda)\otimes \cL(\mu)\]
\[\te_i(F_i^{(n)}u\otimes F_i^{(m)}u')\in \cL(\lambda)\otimes \cL(\mu)\]
This is a rank 1 calculation, and follows from the tensor product rule
(Theorem \ref{thm:tensorproductrule})

(ii)-(v). These follow immediately from Part (i), Lemma \ref{Ka:4.3.1},
and Theorem \ref{thm:tensorproductrule}.
\end{proof}

\begin{cor}[cf. {\cite[Corollary 4.3.5]{Ka}}]\label{Ka:4.3.5}
For $i_1,\ldots, i_l\in I$ and $\mu\in P^+$, let $\lambda=\omega_{l-1}$. Then
\[\tf_{i_1}\ldots \tf_{i_l}(v^+_\lambda\otimes v^+_\mu)+q\cL(\lambda)\otimes \cL(\mu)=v\otimes w+q\cL(\lambda)\otimes \cL(\mu)\]
where $v\in B(\lambda)^{\lambda+\zeta}$ and $w\in B(\mu)^{\mu+\zeta'}\cup\set 0$
for some $\zeta,\zeta'\in Q^-(l-1)\setminus\set 0$.
\end{cor}

\begin{proof}
Assume that $i_l\neq i_{l-1}$. Then $\tf_{i_l}v^+_\lambda=0$,
so \[\tf_{i_l}(v^+_\lambda\otimes v^+_\mu)=f_{i_l}(v^+_\lambda\otimes v^+_\mu)=q_{i_l}^{\langle\alpha_{i_l},
\omega_{i_{l-1}}\rangle} v^+_\lambda\otimes f_{i_{l}}v^+_\mu= v^+_\lambda\otimes \tf_{i_l} v^+_\mu.\]
Then $\te_{i_{l-1}}\tf_{i_l} v^+_\mu=0$ and
$\tf_{i_{l-1}} v^+_\lambda=f_{i_{l-1}}v^+_\lambda\neq 0$ whence
\[\tf_{i_{l-1}}\tf_{i_{l}}(v^+_\lambda\otimes v^+_\mu)
=(\tf_{i_{l-1}}v^+_\lambda)\otimes(\tf_{i_l}v^+_\mu)
\text{ mod }q\cL(\lambda)\otimes \cL(\mu)\]
If $i_{l-1}=i_{l}$, then since $\tf_{i_l}^2 v^+_\lambda=0$,
\[\tf_{i_{l}}^2(v^+_\lambda\otimes v^+_\mu)
=\pi_{i_l}^{p(i_{l})}(\tf_{i_{l}}v^+_\lambda)\otimes(\tf_{i_l}v^+_\mu)
\text{ mod }q\cL(\lambda)\otimes \cL(\mu)\]
Then by Lemma \ref{Ka:4.3.2}(ii), the assertion follows.
\end{proof}

\begin{cor}[cf. {\cite[Corollary 4.3.6]{Ka}}]\label{Ka:4.3.6}
Let $\lambda,\mu\in P^+$ and $\zeta\in Q^-(l)$. Then
\[(\cL(\lambda)\otimes \cL(\mu))^{\lambda+\mu+\zeta}=\sum_{i}
\tf_i(\cL(\lambda)\otimes \cL(\mu))^{\lambda+\mu+\zeta+\alpha_i}+v^+_\lambda\otimes \cL(\mu)^{\mu+\zeta}.
\]
\end{cor}

\begin{proof}
Let $\cL$ denote the left-hand side and $\tilde{\cL}$ denote the right-hand side of the above desired identity.
It is clear that $\tilde{\cL}\subseteq \cL$.

For $\zeta'\in Y^-(l-1)\setminus\set 0$
and $b\in B(\lambda)^{\lambda+\zeta'}\otimes B(\mu)^{\mu+\zeta-\zeta'}$, then
there exists $i\in I$ with $\te_i b\neq 0$ by Lemma \ref{Ka:4.3.2}(iv), whence
$b=\tf_i\te_ib$ by part (iii).
Therefore,
$\cL(\lambda)^{\lambda+\zeta'}\otimes \cL(\mu)^{\mu+\zeta-\zeta'}\subset \tilde{\cL}+q\cL$
and thus
\[
\cL\subset \tilde{\cL}+\cL(\lambda)\otimes v^+_\mu+q\cL.
\]
Now, for $\tf_{i_1}\ldots\tf_{i_l}v^+_\lambda\in B(\lambda)^{\lambda+\zeta}$,
we have
\[(\tf_{i_1}\ldots\tf_{i_l}v^+_\lambda)\otimes v^+_\mu=
\tf_{i_1}((\tf_{i_2}\ldots\tf_{i_l}v^+_\lambda)\otimes v^+_\mu)
\text{ mod }q\cL(\lambda)\otimes \cL(\mu)\]
and hence $\cL\subset \tilde{\cL}+q\cL$. But $q\in \Rad(\Ap)$,
the Jacobson radical of $\Ap$, hence by Nakayama's lemma, $\cL=\tilde{\cL}$.
\end{proof}

\begin{lem}[compare {\cite[Lemma 4.7.1]{Ka}}]\label{Ka:4.7.1}
For $\zeta=-\sum n_i\alpha_i\in Q^-$ and $P,Q\in \UU^-$, there exists
a polynomial $f(x_1,\ldots, x_n)$ in $x=(x_i)_{i\in I}$ with coefficients
in $\Qqp$ such that
\begin{equation}
(Pv^+_\lambda, Qv^+_\lambda)=f(x)\quad\text{ with }
x_i=(\pi_iq_i^2)^{\ang{\alpha_i^\vee,\lambda}},
\text{ and}
\end{equation}
\begin{equation}
f(0)=\parens{\prod (1-\pi_i q_i^2)^{-n_i}}(P,Q).
\end{equation}
\end{lem}

For a weight $\lambda\in P^+$, we write that
$\lambda\gg 0$ if $\lambda-\zeta\geq 0$ for all
$\zeta\in Q^-(l)$.

\begin{lem}[{\cite[Lemma 4.7.2]{Ka}}]\label{Ka:4.7.2}
For $\mu\gg 0$, $\wp_\mu(\cL(\infty)^{\zeta,\vee})=\cL(\mu)^{\mu+\zeta,\vee}$
for $\zeta\in Q^-(l)$.
\end{lem}

Recall the comultiplication $\Delta'$ from \eqref{eq:comult'} was used
in defining the operators $\Phi'$ and $\Psi'$, just as  $\Delta$  was used
in defining the operators $\Phi$ and $\Psi$; see \eqref{eqn:tensorembeddings} and Lemma~\ref{lem:pseudopol}.

\begin{lem}[compare {\cite[Proposition 4.7.3]{Ka}}]\label{Ka:4.7.3}
Let $\lambda,\mu\in P^+$ with $\mu\gg 0$ and let $\zeta\in Q^-(l)$.
Then
\begin{align*}
\Psi(\lambda,\mu)( (\cL(\lambda)\otimes \cL(\mu) )^{\lambda+\mu+\zeta})
&= \cL(\lambda+\mu)^{\lambda+\mu+\zeta}.
\end{align*}
\end{lem}

\begin{proof}
In this proof, we shall need $\Phi'(\lambda,\mu)$ instead of $\Phi(\lambda,\mu)$ as in \cite{Ka}.

We have
$\wp_{\lambda+\mu}(\cL(\infty)^\zeta)=\cL(\lambda+\mu)^{\lambda+\mu+\zeta}$
and $\wp_\mu(\cL(\infty)^{\zeta,\vee})=\cL(\mu)^{\mu+\zeta,\vee}$.
We have
\[(\cL(\lambda)\otimes \cL(\mu) )^{\lambda+\mu+\zeta}
=\sum \tf_i(\cL(\lambda)\otimes \cL(\mu))^{\lambda+\mu+\zeta+\alpha_i}+v^+_\lambda\otimes \cL(\mu)^{\mu+\zeta}.\]

Now for $u\in \cL(\lambda+\mu)^{\lambda+\mu+\zeta,\vee}$, we have
\begin{align*}
(\Phi'(\lambda,&\mu)(u),\tf_i(\cL(\lambda)\otimes \cL(\mu))^{\lambda+\mu+\zeta+\alpha_i})\\
&=(u, \tf_i\Psi(\lambda,\mu)(\cL(\lambda)\otimes \cL(\mu))^{\lambda+\mu+\zeta+\alpha_i})\\
&\subset (u, \tf_i \cL(\lambda+\mu)^{\lambda+\mu+\zeta+\alpha_i})\subset \A.
\end{align*}

Let $u=Pv^+_{\lambda+\mu}$ with $P\in \cL(\infty)^{\zeta,\vee}$. Then for
$\zeta=-\sum n_i \alpha_i$, we have

\[\Delta'(P)=\parens{\prod (\tJ_i\tK_i)^{n_i}} \otimes P \text{ mod }\parens{\sum F_i \UU^{\leq 0}}\otimes \UU^-.\]

Therefore we have

\[\Phi'(\lambda,\mu)(Pv^+_{\lambda+\mu})=\parens{\prod (\pi_i q_i)^{n_i\ang{\alpha_i^\vee, \lambda}}}
 v^+_\lambda\otimes Pv^+_\mu\text{ mod } \parens{\sum F_i V(\lambda)}
\otimes V(\mu)\]

and thus
\[(\Phi'(\lambda,\mu)(u), v^+_\lambda\otimes \cL(\mu))\subseteq\parens{\prod (\pi_i q_i)^{n_i\ang{\alpha_i^\vee, \lambda}}}
 (Pv^+_\mu, \cL(\mu)^{\mu+\zeta})\subset \A.
\]

So we have shown
\[
(\cL(\lambda+\mu)^{\lambda+\mu+\zeta,\vee},\Psi(\lambda,\mu)(\cL(\lambda)\otimes \cL(\mu))^{\lambda+\mu+\zeta})\subset \A
\]
and thus $\Psi(\lambda, \mu)(\cL(\lambda)\otimes \cL(\mu)^{\lambda+\mu+\zeta})
\subset \cL(\lambda+\mu)^{\lambda+\mu+\zeta}$.
Since $\Psi(\lambda,\mu)\circ \Phi(\lambda,\mu)$ is the identity,
the reverse inclusion also holds.
\end{proof}

The other parts of Kashiwara's grand loop argument works  equally well in our current setting.
Summarizing, Kashiwara's grand loop argument in \cite[\S4]{Ka} with the above modifications
goes through, and we have established ($C_l.1$)-($C_l.14$).

Clearly the validity of ($C_l.1$)-($C_l.14$) implies Theorems~\ref{th:CrBinf}-\ref{th:CrBinfla}
in \S \ref{sec:thmcrystal} for $\diamond =-$.
Also these three theorems in \S \ref{sec:thmcrystal} hold for $\diamond =+$ by \cite{Ka}.
The crystal bases involved in Theorems~\ref{th:CrBinf}-\ref{th:CrBinfla} are formally defined by the same formulas.
Hence we can lift the crystal bases in the cases $\diamond=\pm$ to a formal parameter $\pi$ with $\pi^2=1$,
and we conclude that Theorems~\ref{th:CrBinf}-\ref{th:CrBinfla} hold in the case $\diamond =\emptyset$ as well.

\section{Properties of polarization} \label{sec:proppol}

Let us  examine more closely the properties of the polarizations
on $\cL(\lambda)$ and   on $\cL(\infty)$.

 \subsection{The $\pi$-orthonormality of polarizations}

Recall the notion of ($\pi$-orthonormal) $\pi$-basis from Definition~\ref{dfn:pibasis}.

\begin{prop}\label{prop:polprops}
For $\diamond \in \{+,-\}$, let either  $(1)$ $V =\UU^-_\diamond$ and $(\cL,B)= (\cL(\infty)_\diamond, B(\infty)_\diamond)$
or  $(2)$ $V=V(\la)_\diamond$ and $(\cL,B)= (\cL(\lambda)_\diamond,B(\lambda)_\diamond)$ for $\lambda\in P^+$. Let $(\cdot,\cdot)$
be the polarization on $V$ given in \S \ref{subsec:bilformU-} or \S \ref{subsec:polar}, accordingly.
Then,
\begin{enumerate}
\item[(i)] $(\cL,\cL)\subseteq \A$, and so it descends to a bilinear form
            \[(\cdot,\cdot)_0:\cL/q\cL\times \cL/q\cL\rightarrow \Q,\quad (x+q\cL,y+q\cL)_0=(x,y)|_{q=0}.\]

\item[(ii)] $(\tf_i u,v)_0=\pi_i^{\epsilon_i(u)}(u,\te_iv)_0$ for $u,v\in \cL/q\cL$
(Here $\pi$ is understood as $1$ or $-1$ for $\diamond \in \{+,-\}$, respectively).

\item[(iii)] $B$ is a $\pi$-orthonormal
             $\pi$-basis of $\cL/q\cL$ with respect to $(\cdot,\cdot)_0$.

\item[(iv)] $\cL=\set{u\in V; (u,\cL)\subseteq \A}$.
\end{enumerate}
\end{prop}

\begin{proof}
For notational simplicity and certainty,
we will prove the case (2) in detail, while the case (1) is entirely similar.

The same easy reduction as in the proof of  \cite[Proposition~5.1.1]{Ka} reduces the proof of Parts (i) and (ii)
to the verification
of the following identity
\begin{equation}  \label{eq:adjoint}
(\tf_i u, v) \equiv \pi_i^{\epsilon_i(u)} (u, \te_i v) \quad \mod q\A
\end{equation}
where $u=F_i^{(n)}u_0 \in \cL(\la)_\diamond^{\la+\zeta+\alpha_i}, v=F_i^{(m)}v_0 \in \cL(\la)_\diamond^{\la+\zeta}$
with $E_i u_0 =E_iv_0 =0$.

To that end, we have the following computation (compare \cite[(5.1.2)]{Ka}):
\begin{align*}
(F_i^{(n+1)}u_0, F_i^{(m)}v_0)
&=\delta_{n+1,m} q_i^{m\ang{\alpha_i^\vee, \lambda+\zeta}+m^2}(E_i^{(m)}F_i^{(m)}u_0, v_0)\\
&=\delta_{n+1,m} \pi_i^{m^2+\binom{m+1}{2}}q_i^{m\ang{\alpha_i^\vee, \lambda+\zeta}+m^2}
\bbinom{\ang{\alpha_i^\vee, \lambda+\zeta}+2m}{m}(u_0, v_0)\\
&\equiv \delta_{n+1,m}\pi_i^{\binom{m}{2}}(u_0,v_0) \mod q \A,
\end{align*}
where we have used
$\binom{m+1}{2}+m^2\equiv \binom{m}{2} \mod 2.$
Therefore $(\tf_i u, v)_0=\pi_i^{m-1}(u,\te_i v)_0$.
Since $m-1=n=\epsilon_i(u)$, the identity~\eqref{eq:adjoint} follows, and whence (i) and (ii).

Part (iii) follows by induction on weights and using Theorem~\ref{th:CrBla} from the identity
$$(b,b')_0 =(\tf_i \te_i b, b')_0 =(\te_ib, \te_ib')_0,
$$
where $b,b' \in B$ and $i \in I$ is chosen such that $\te_i b \in B$.

To prove (iv), it remains to verify that $\set{u\in V; (u,\cL)\subseteq \A} \subseteq \cL$ thanks to (i).
Denote $\sgn(b)=(b,b)_0$.
Suppose $u\in V$ is a $\mu$-weight vector such that
$(u,\cL)\subseteq \A$. By Theorem~\ref{th:CrBla}  and the definition of crystal basis and $\pi$-basis,
one can find $B_\mu^0 \subset B_\mu$ which is an honest basis for $\cL_\mu/ q\cL_\mu$.
Then $u$ can be written
as $u=\sum_{b\in B^0_\mu} c_b u_b$. where 
 $u_b+q\cL=b$ and $c_b\in \Qq$.
Assume $u\notin \cL$.
Then there  exists a minimal $r\in \Z_{> 0}$
such that $q^r c_b\in \A$ for all $b\in B_\mu$.
Since $(u,\cL)\subset \A$, we have in particular that $(u,\sgn(b)q^{r-1}u_b)\in \A$.
On the other hand, since $(u_b, u_{b'})\in q\A$ for $b\neq b'$,
we compute that $(u, \sgn(b)q^{r-1}u_b)\in q^{r-1}c_b+\A $ for all $b$,
whence $q^{r-1}c_b \in \A$ for all $b$, contradicting
the minimality of $r$. This completes the proof of the proposition.
\end{proof}

\begin{rmk}  \label{rmk:nonchar}
\begin{enumerate}
\item
One can formulate a version of Proposition~\ref{prop:polprops} with $\diamond =\emptyset$
and the bilinear form $(\cdot,\cdot)_0$ taking value in $\Qp$.

 \item
In contrast to the usual quantum group setting in \cite{Ka, Lu2}, $(\cdot,\cdot)_0$ here is not positive
definite in general, as it could happen that $(b,b)_0 =\pi$ for some crystal basis element.
In particular, the well-known characterization in the usual quantum group setting
that an element $u$ lies in the crystal lattice if and only if $(u,u)\in \A$ fails in our super setting in general;
see Example~\ref{ex:nonchar}.
\end{enumerate}
\end{rmk}

\begin{example}  \label{ex:nonchar}
Let $\UU$ be the quantum $\osp(1|4)$,
with $\alpha_1$ the short root and $\alpha_2$ the long root.
Then $\tf_1^4\tf_2\cdot 1=F_1^{(4)}F_2$ and
$\tf_1^{3}\tf_2\tf_1\cdot 1=F_1^{(3)}(F_2F_1-q^2F_1F_2)+q^2 F_1^{(4)}F_2$
(these will be canonical basis elements as developed in Section~\ref{sec:CB}).
A direct computation shows that
$$
\big(F_1^{(3)} (F_2F_1-q^2F_1F_2), F_1^{(4)}F_2 \big) =0,
$$
and also
\begin{align*}
\big(F_1^{(4)}  F_2,F_1^{(4)}F_2 \big)
&=(\pi q)^6([4]^!)^{-1} \in 1+q^2\Z^\pi [[q]],
\\
\big(F_1^{(3)}  (F_2F_1-q^2F_1 & F_2), F_1^{(3)}(F_2F_1-q^2F_1F_2) \big)
 \\
&=(\pi q)^3([3]^!)^{-1}(1-q^4) \in \pi+q^2\Z^\pi [[q]].
\end{align*}
It follows that
\begin{enumerate}
\item $\big(\tf_1^4\tf_2\cdot 1, \tf_1^4\tf_2\cdot 1 \big)=1\mod q^2\Z^\pi [[q]]$;
\item $\big(\tf_1^3\tf_2\tf_1\cdot 1, \tf_1^3\tf_2\tf_1\cdot 1\big)=\pi \mod q^2\Z^\pi [[q]]$;
\item $(\tf_1^4\tf_2\cdot 1, \tf_1^3\tf_2\tf_1\cdot 1)=q^2 \mod q^4\Z^\pi [[q]]$.
\end{enumerate}

Now (1) and (2) provide us an example that {\em the squared norm of (canonical basis) elements in $B(\infty)$ of the same weight
do not have uniform sign}. Combined with (3), this implies that $\cL(\infty)\subsetneq \set{u\in U^-|(u,u)\in \A}$,
since $q^{-1}(1-\pi)(\tf_1^4\tf_2\cdot 1 + \tf_1^3\tf_2\tf_1\cdot 1)$ belongs to the right-hand side, but not to $\cL(\infty)$.
\end{example}

 \subsection{Polarization and $\varrho$}

 We simply formulate the counterpart of \cite[\S 5.2]{Ka}.

 Recall from \cite[\S 2.1.2]{CHW}
 an algebra automorphism $\omega$ on $\UU$ and an algebra anti-involution $\sigma$ on $\UU$ (i.e. $\sigma(xy)=\sigma(y)\sigma(x)$).
 Note that $\sigma$ fixes each $E_i$ but not $F_i$ (a super phenomenon).
 In particular $\varrho:=\omega \sigma \omega^{-1}$ is a (non-super) algebra anti-involution on $\UU$ such that
\begin{equation}  \label{eq:anti inv}
 \varrho (F_i) =F_i, \quad \varrho(E_i)=\pi_i \tJ_i E_i, \quad \varrho(K_\mu)=K_{-\mu}, \quad \varrho(J_\mu) =J_\mu
 \quad (i\in I, \mu \in P^\vee).
\end{equation}
We start with some lemmas.

\begin{lem}
For $i,j \in I$ and $P,Q \in \UU^-$, we have
\begin{align}
 \label{eq:ELR}
 (\Ad( \tK_i) E_i'') \circ E_j' & =\pi^{p(i)p(j)} E_j' \circ (\Ad(\tK_i) E_i''),
 \\
\label{eq:Radj}
 (PF_i, Q)  &=  \pi_i^{p(P)} (P, \Ad(\tK_i) E_i'' Q).
\end{align}
\end{lem}

\begin{proof}
The identity \eqref{eq:ELR} is a reformulation of the $q$-commutativity of $E_i''$ and $E_j'$ in Proposition~\ref{prop:qcommute}.
Then \eqref{eq:Radj} is proved in the same way as in \cite[Lemma 5.2.2]{Ka}, modified by \eqref{eq:ELR} above.
\end{proof}

We shall also write $\varrho(Q)$ as $Q^\varrho$ for $Q \in \UU^-$.
\begin{lem}
We have $\big( E_i' (Q^\varrho) \big)^\varrho = \pi_i^{p(Q)+1} \Ad (\tK_i)E_i''Q$, for $Q \in \UU^-$.
\end{lem}

\begin{proof}
We can rewrite \eqref{eq:E'} as
$$
[E_i, Q] =\frac{\tJ_i (\Ad (\tK_i) E_i''Q) \tK_i-((\Ad \tK_i^{-1}) E_i'Q) \tK_i^{-1}}{\pi_i q_i-q_i^{-1}}.
$$
Applying $\varrho$ (see \eqref{eq:anti inv}) to the above identity gives us
$$
[Q^\varrho, E_i^\varrho] =\frac{\tK_i^{-1} \tJ_i (\Ad (\tK_i) E_i''Q)^\varrho - \tK_i ((\Ad \tK_i^{-1}) E_i'Q)^\varrho}{\pi_i q_i-q_i^{-1}}.
$$
Using $[Q^\varrho, E_i^\varrho] =-\pi_i^{p(Q)}[\pi_i \tJ_i E_i, Q^\varrho]$ and noting $\tJ_i$ is central, we rewrite the above as
$$
[E_i, Q^\varrho] = \pi_i^{p(Q)+1}
\frac{\tJ_i\tK_i ( (\Ad \tK_i^{-1}) E_i'Q)^\varrho -\tK_i^{-1} (\Ad (\tK_i) E_i''Q)^\varrho}{\pi_i q_i-q_i^{-1}}.
$$
A comparison with \eqref{eq:E'} (by setting $y=Q^\varrho$) gives us the desired formula.
\end{proof}

Now we are ready to show that  $\varrho$ is an isometry (without signs).
\begin{prop}
\label{prop:isometry}
For $P, Q \in \UU^-$, we have
\begin{equation}  \label{eq:isometry}
(P^\varrho, Q^\varrho) =(P,Q).
\end{equation}
\end{prop}

\begin{proof}
The claim is trivial when $P=1$, and so it suffices to prove that the identity \eqref{eq:isometry} implies that
$$
((PF_i)^\varrho, Q^\varrho) =(PF_i,Q).
$$
We will assume without loss of generality that $p(Q)=p(PF_i) =p(P)+1$, as otherwise both sides are clearly equal to $0$.

\begin{align*}
((PF_i)^\varrho, Q^\varrho)
&= (F_i P^\varrho, Q^\varrho) =(P^\varrho, E_i' Q^\varrho)
 \\
&= (P, (E_i' Q^\varrho) ^\varrho) = \pi_i^{p(Q)+1} (P,   \Ad (\tK_i)E_i''Q)
 \\
&= \pi_i^{p(Q)+1} \pi_i^{p(P)} (P F_i,   Q) =(P F_i, Q).
\end{align*}
This completes the proof.
\end{proof}

The fact
that $\varrho (\cL(\infty)_+) =\cL(\infty)_+$ follows easily from \eqref{eq:isometry} and the orthonormality characterization of crystal
lattice  in the standard quantum group setting \cite{Lu1, Lu2, Ka}.
While such orthonormality characterization fails in our super setting as noted in Remark~\ref{rmk:nonchar}(2),
 the $\varrho$-stable property of the crystal lattice remains to be true.  
 
\begin{prop}    \label{prop:L*=L}
For $\diamond \in \{\emptyset, -\}$, we have $\varrho (\cL(\infty)_\diamond) =\cL(\infty)_\diamond.$
\end{prop}

\begin{rmk}
Proposition~\ref{prop:L*=L} is proved in \cite[Proposition~3.4]{CFLW},
as it is an immediate consequence of connections between $\UU$ and the Drinfeld-Jimbo quantum groups developed therein. 
\end{rmk}

 \subsection{A comparison of polarizations}

Now let us compare the bilinear forms on $\cL(\lambda)$ and $\cL(\infty)$.

\begin{prop} \label{prop:pollimitKa}
For given $x,y\in \cL(\infty)_\zeta$, take $\lambda\gg 0$. Then
$(xv^+_\lambda, yv^+_\lambda)_0= (x,y)_0$.
\end{prop}

\begin{proof}
This is obvious for $\zeta=0$. We proceed by induction on the height of $\zeta$.
We can write $x=F_i x'$ for some $i\in I$.
Then
\begin{align*}
(xv^+_\lambda, yv^+_{\lambda})
&=(x'v^+_\lambda, q_i^{-1} \tK_i E_i y v^+_\lambda)\\
&= \Big(x'v^+_\lambda, \frac{\tJ_i\tK_i^2E_i''(y) -E_i'(y)}{\pi_i q_i^2-1}v^+_\lambda \Big)\\
&=\frac{(\pi_iq_i^2)^{\ang{\alpha_i^\vee,\lambda+|y|+\alpha_i}}}{\pi q_i^2-1}(x'v^+_{\lambda},
E_i''(y)v^+_\lambda)+\frac{1}{1-\pi_i q_i^2}(x' v^+_{\lambda}, E_i'(y)v^+_\lambda).
\end{align*}
Hence, by induction and the assumption $\lambda\gg 0$, we have
\begin{align*}
(xv^+_\lambda, yv^+_{\lambda})_0
=(x' v^+_{\lambda}, E_i'(y)v^+_\lambda)_0
=(x', E_i'(y))_0=(x,y)_0.
\end{align*}
The proposition is proved.
\end{proof}

We can also relate these bilinear forms to the bilinear form on
$\UU^-$ given in \cite{CHW}. We recall from that paper the following notation.
The algebra $\ff$ is the $\Qqp$-algebra generated by $\theta_i$
such that there is an isomorphism
$\ff\cong \UU^-$ given by $\theta_i\mapsto \theta_i^-=F_i$.
Via this identification, the twisted derivations $r_i$ and ${}_ir$ on $\ff$ lead
to twisted derivation on $\UU^-$ denoted by the same notations.
There is an bilinear form $(\cdot, \cdot)$ on $\ff$ and hence on $\UU^-$ \cite[\S 1]{CHW} which will be
denoted by $(\cdot, \cdot)_L$ here to avoid conflict of notation,
and its bar conjugate $\set{\cdot, \cdot}$ defined by
$\set{x,y}=\bar{(\bar x, \bar y)_L}$, for $x,y \in \UU^-$.

By comparing the formula for $E_i'$ in Lemma~\ref{lem:E'Fcomms}
and the formula for $r_i$ in \cite[\S1.3.13]{CHW}, we have
by
$$
E_i'(y)=\pi_i^{p(y)-p(i)}q_i^{\ang{\alpha_i^\vee, |y|+\alpha_i}}r_i(y), \quad \text{ for } y\in \UU^-.
$$
By \cite[Lemma~1.3.14]{CHW},
$
\ri( y)=\pi_i^{p(y)-p(i)} q_i^{\ang{\alpha_i^\vee, -|y|-\alpha_i}} \bar{\ir(\bar y)},
$
and hence we obtain
\begin{equation}  \label{eq:E=r}
E_i'(y)=\bar{\ir(\bar y)}.
\end{equation}

\begin{prop}
Let $x,y\in \UU^-$.
Then
$\lim_{\la\mapsto \infty} (x^-v^+_\lambda, y^-v^+_\lambda)= \set{x,y}$ in the $q$-adic norm.
\end{prop}

\begin{proof}
Recall \cite[\S1]{CHW} a defining property of $(\cdot,\cdot)_L$ is that
\[(\theta_i, \theta_i)_L (x,\ir(y))_L=(\theta_ix,\theta_iy)_L.\]

We have $(\pi_iq_i^2)^{\ang{\alpha_i^\vee,\lambda+|y|+\alpha_i}}\rightarrow 0$ in the $q$-adic norm as
$\lambda\mapsto  +\infty$; so using the computations in the proof of
Proposition \ref{prop:pollimitKa}, we find that
\begin{align*}
\lim_{\la\mapsto \infty} (xv^+_\lambda, yv^+_{\lambda}) &=
\frac{1}{1-\pi_i q_i^2}\set{x' , E_i'(y)}
=\set{\theta_i, \theta_i}\set{x', \bar{\ir(\bar y)}}\\
&=\bar{(\theta_i, \theta_i)_L (\bar{x'}, \ir(\bar y))_L}
=\bar{(\theta_i \bar{x'}, \bar y)_L}=\set{x,y}.
\end{align*}
The proposition is proved.
\end{proof}


\section{Canonical bases}
\label{sec:CB}

In this section we shall prove the existence of canonical bases (= global crystal bases)
for $\UU^-$ and all integrable modules.

\subsection{The integral form of $\UU$}

Let $\UUZ$ be the $\Z[q,q^{-1}]$-subalgebra of $\UU$ generated by $F_i^{(n)}$,
$E_i^{(n)}$, $J_\mu$, $K_\mu^{\pm 1}$ and $\bbinom{K_i;0}{a}$ for all $i\in I$,
$n\in \Z_{\geq 0}$, and $a\in \Z$. We set $\UUZ^-$ to be the
$\Z[q,q^{-1}]$-subalgebra generated by $F_i^{(n)}$ for $i\in I$ and $n\in \Z_{\geq 0}$.
Then $\UUZ$ and $\UUZ^-$
are stable under the bar involution \eqref{eqn:barinv}.
Moreover, $\UUZ^-$ is stable under $E_i'$
whence $\UUZ^-$ is stable under  Kashiwara operators $\te_i$ and $\tf_i$; therefore
\[u=\sum_{n\geq 0} F_i^{(n)} u_n \in \UUZ^-\text{ and }e_i'u_n=0\ \Longrightarrow
u_n\in \UUZ^-.\]
Let $(F_i^n \UU^-)_\Z=F_i^n\UU^-\cap \UUZ^-$.
Then
\begin{equation*}
 (F_i^n\UU^-)_\Z=\sum_{k\geq n} F_i^{(k)}\UUZ^-, \qquad \text{for }n\geq0.
\end{equation*}
Moreover, $u=\sum F_i^{(k)} u_k\in(F_i^n\UU^-)_\Z$ if and only if $u_k=0$
for $k<n$.
Set $\cL_\Z(\infty)=\cL(\infty)\cap \UUZ^-$. Then $\cL_\Z(\infty)$ is stable
under the Kashiwara operators $\te_i$ and $\tf_i$. Therefore, $B(\infty)\subset \cL_\Z(\infty)/q\cL_\Z(\infty)\subset \cL(\infty)/q\cL(\infty)$.

Similarly, for $\diamond \in \{+,-\}$, we can define the integral forms $\UUZ_\diamond$ and $\UUZ_{\diamond}^-$
of the specializations $\UU_\diamond$ and $\UU_\diamond^-$, respectively. Then
$\cL_\Z(\infty)_\diamond :=\cL(\infty)_\diamond \cap \UUZ^-_\diamond$ is stable under Kashiwara operators.

Let $\A_\Z$
be the $\Zp$-subalgebra of $\Qqp$ generated by $q$ and $(1-(\pi q^2)^n)^{-1}$
for $n\geq 1$. Similarly, for $\diamond=+,-$,  let $\A_\Z^\diamond$
be the $\Z$-subalgebra of $\Qqp$ generated by $q$ and $(1-(\diamond q^2)^n)^{-1}$
for $n\geq 1$.
Letting $K_\Z^\diamond$ be the subalgebra generated by $\A_\Z^\diamond$
and $q^{-1}$, we have $\A_\Z^\diamond =\A\cap K_\Z^\diamond$, for $\diamond \in \{\emptyset, +,-\}$.
(As before, the superscript or subscript $\emptyset$ is dropped by convention.)
Then we see that $(\UUZ^-_\diamond, \UUZ^-_\diamond)\subset K_\Z^\diamond$,
whence $(\cL_\Z(\infty)_\diamond, \cL_\Z(\infty)_\diamond)\subset\A_\Z^\diamond$. Therefore,
$(\cdot,\cdot)_0$ is $\Zp$-valued on
$\cL_\Z(\infty)/q\cL_\Z(\infty)$, and  $\cL_\Z(\infty)/q\cL_\Z(\infty)$
is a free $\Zp$-module with $\pi$-basis $B(\infty)$. Similarly, for $\diamond=+,-$, $(\cdot,\cdot)_0$ is $\Z$-valued on
$\cL_\Z(\infty)_\diamond/q\cL_\Z(\infty)_\diamond$, and  $\cL_\Z(\infty)_\diamond/q\cL_\Z(\infty)_\diamond$
is a free $\Z$-module with $\pi$-basis $B(\infty)_\diamond$.

\subsection{The integral form of $V(\lambda)$}

Let $\diamond \in \{\emptyset, +, -\}$.
Set $V_\Z(\lambda)_\diamond=\UUZ^-_\diamond v^+_\lambda$. Then $V_\Z(\lambda)_\diamond$ is a $\UUZ_\diamond$-module.
We set also, for $n\geq 0$,
\begin{equation*}
(F_i^n V(\lambda)_\diamond)_\Z=(F_i^n\UU^-_\diamond)_\Z v^+_\lambda=\sum_{k\geq n} F_i^{(k)} V_\Z(\lambda)_\diamond.
\end{equation*}
Note that $V_\Z(\lambda)_\diamond$ and $(F_i^n V(\lambda)_\diamond)_\Z$ are bar-invariant.

Let $\cL_\Z(\lambda)_\diamond=V_\Z(\lambda)_\diamond\cap \cL(\lambda)_\diamond$.
Since $\cL(\lambda)_\diamond=\wp_\lambda(\cL(\infty)_\diamond)$ we have
\begin{equation*}
\cL_\Z(\lambda)_\diamond\subset \wp_\lambda(\cL_\Z(\infty)_\diamond)
\end{equation*}
and so $B(\lambda)_\diamond\subset \cL_\Z(\lambda)_\diamond/q\cL_\Z(\lambda)_\diamond\subset \cL(\lambda)_\diamond/q\cL(\lambda)_\diamond$.

The following is a $\pi$-analogue of \cite[Lemma~6.1.14]{Ka}.

\begin{lem}
\label{lem:divpowdecomp}
Let $\lambda\in P$, $i\in I$, and $u\in M^\lambda$ for an integrable
$\UU$-module $M$.
Assume $n=-\ang{\alpha_i^\vee,\lambda}\geq 1$. Then we have
\[u=\pi_i^{\binom{n}{2}}\sum_{k\geq n} (-1)^{k-n}\bbinom{k-1}{k-n} F_i^{(k)}E_i^{(k)} u.\]
\end{lem}

\begin{proof}
We may assume $u=F_i^{(m)}v$ with $v\in \ker e_i \cap M_{\lambda+m \alpha_i}$
with $m\geq n$. Then

\begin{align*}
\sum_{k\geq n} (-1)^{k-n}&\bbinom{k-1}{k-n}_i F_i^{(k)}E_i^{(k)} u\\
&=\sum_{k\geq n} (-1)^{k-n}\bbinom{k-1}{k-n}_i F_i^{(k)}E_i^{(k)} F_i^{(m)}v\\
&=\sum_{m\geq k\geq n} \pi_i^{km+\binom{k+1}{2}}(-1)^{k-n}\bbinom{k-1}{k-n}_i \bbinom{k+m-n}{k}_iF_i^{(k)}F_i^{(m-k)}v\\
&= \sum_{m\geq k\geq n} \pi_i^{km+\binom{k+1}{2}}(-1)^{k-n}\bbinom{k-1}{k-n}_i \bbinom{k+m-n}{k}_i\bbinom{m}{k}_i \cdot F_i^{(m)}v.
\end{align*}
By a change of variables with $t=k-n$ and $r=m-n$, we have
\begin{align*}
&\sum_{m\geq k\geq n} \pi_i^{km+\binom{k+1}{2}}(-1)^{k-n}\bbinom{k-1}{k-n}_i \bbinom{k+m-n}{k}_i\bbinom{m}{k}_i\\
&=\sum_{t=0}^r (-1)^{t}\pi_i^{(t+n)(r+n)+\binom{t+n+1}{2}}\bbinom{t+n-1}{t}_i \bbinom{t+r+n}{r}_i\bbinom{r+n}{t+n}_i.
\end{align*}

The proof is completed by \cite[(6.1.19)]{Ka} if $i\in I_0$ and by the following lemma
if $i\in I_1$.
\end{proof}

\begin{lem} We have the following identity for $r\ge 0$ and $n \ge 1$:
\begin{equation*}
\sum_{t=0}^r (-1)^{t}\pi^{(t+n)(r+n)+\binom{t+n+1}{2}}\bbinom{t+n-1}{t} \bbinom{t+r+n}{r}\bbinom{r+n}{t+n}=\pi^{\binom{n}{2}}.
\end{equation*}
\end{lem}

\begin{proof}
We first introduce the following notations. Let
\[\ang{n}:=[n]|_{\pi=1}=\frac{q^n-q^{-n}}{q-q^{-1}},\quad \set{n}_x:=\frac{1-x^n}{1-x}.\]
We extend this notation to factorials
and binomials in a self-explanatory manner.
We also note that
\begin{align}
\label{eqn:pqinttogauss}[n]&=(\pi q)^{n-1}\set{n}_{\pi q^{-2}},\\
\label{eqn:qinttogauss}\ang{n}&=q^{n-1}\set{n}_{q^{-2}}.
\end{align}

The identity \cite[(6.1.19)]{Ka} may be stated in this notation as
\[
\sum_{t=0}^r (-1)^{t}\abinom{t+n-1}{t}\abinom{t+r+n}{r}\abinom{r+n}{t+n}=1,
\]
which can be transformed using \eqref{eqn:qinttogauss} into
\[
\sum_{t=0}^r (-1)^{t}(q^{-2})^{\binom{t+1}{2}-(n+t)r}\cbinom{t+n-1}{t}_{q^{-2}} \cbinom{t+r+n}{r}_{q^{-2}}\cbinom{r+n}{t+n}_{q^{-2}}=1.
\]
Note that this is a general identity for the $\set{n}_x$, so
\[
\sum_{t=0}^r (-1)^{t}x^{\binom{t+1}{2}-(n+t)r}\cbinom{t+n-1}{t}_x \cbinom{t+r+n}{r}_x\cbinom{r+n}{t+n}_x=1.
\]
In particular, setting $x=\pi q^{-2}$
and using \eqref{eqn:pqinttogauss}, we obtain
\[
\sum_{t=0}^r (-1)^{t}\pi^{\binom{t+1}{2}-(n+t)r}\bbinom{t+n-1}{t} \bbinom{t+r+n}{r}\bbinom{r+n}{t+n}=1.
\]
The lemma follows since $\binom{t+1}{2}-(n+t)r=(t+n)(r+n)+\binom{t+n+1}{2} + \binom{n}{2} \mod  2$.
\end{proof}

\begin{prop}
Let $M$ be an integrable $\UU$-module and $M_\Z$ a $\UUZ$-weight submodule of $M$.
Let $\lambda\in P$ and $i\in I$. Suppose $n=-\ang{\alpha_i^\vee,\lambda}\geq 0$.
Then $$
(M_\Z)^\lambda=\sum_{k\geq n} F_i^{(k)} (M_\Z)^{\lambda+k\alpha_i}.
$$
\end{prop}

\begin{proof}
The lemma implies
that $(M_\Z)^\lambda\subseteq \sum_{k\geq n} F_i^{(k)} (M_\Z)^{\lambda+k\alpha_i}$.
The reverse inclusion is clear.
\end{proof}

\subsection{Existence of canonical bases}

Let us consider the following collection $(G_l)$ of statements for $l\geq 0$, where $\diamond \in \{\emptyset, +, -\}$.
\begin{enumerate}
\item[$(G_l.1)$] For any $\zeta\in Q^-(l)$,
\[(\UUZ^-_\diamond)^\zeta\cap \cL_\Z(\infty)_\diamond \cap \bar{\cL_\Z(\infty)_\diamond}\to \cL_\Z(\infty)_\diamond^\zeta/q\cL_\Z(\infty)_\diamond^\zeta\]
is an isomorphism.
\item[$(G_l.2)$] For any $\zeta\in Q^-(l)$,
\[V_\Z^-(\lambda)_\diamond^{\lambda+\zeta}\cap \cL_\Z(\lambda)_\diamond\cap \bar{\cL_\Z(\lambda)_\diamond}\to \cL_\Z(\lambda)_\diamond^{\lambda+\zeta}/q\cL_\Z(\lambda)_\diamond^{\lambda+\zeta}\]
is an isomorphism.
\end{enumerate}
Let $G, G_\lambda$ be the inverses of these isomorphisms.
\begin{enumerate}
\item[$(G_l.3)$] For any $\zeta\in Q^-(l)$,  $n\geq 0$, and $b\in \tf_i^n(B(\infty)_\diamond^{\zeta+n\alpha+i})$,
\[G(b)\in f_i^n \UU^-.
\]
\end{enumerate}

The case $\diamond =+$ is as in \cite[\S6-7]{Ka}, and so let us now consider the case $\diamond=-$.

Note that when $l=0$, these statements are obvious. We shall
prove $G_l$ by induction on $l$, so assume $l>0$ and $G_{l-1}$ holds.

\begin{lem}
For $\zeta\in Q^-(l-1)$ we have
\begin{align*}
\displaystyle(\UUZ^-_-)^\zeta\cap \cL_\Z(\infty)_- &= \bigoplus_{b\in B(\infty)_-^\zeta} \Z[q] G(b), \\
\displaystyle(\UUZ^-_-)^\zeta&= \bigoplus_{b\in B(\infty)_-^\zeta} \Z[q,q^{-1}] G(b), \\
\displaystyle V_\Z(\lambda)_-^{\lambda+\zeta}\cap \cL (\lambda)_- &= \bigoplus_{b\in B(\lambda)_-^{\lambda + \zeta}} \Z[q] G_\lambda(b), \\
\displaystyle V_\Z(\lambda)_-^{\lambda+\zeta} &= \bigoplus_{b\in B(\lambda)_-^{\lambda + \zeta}} \Z[q,q^{-1}] G_\lambda(b).
\end{align*}
\end{lem}

\begin{proof}
This follows from \cite[Lemma 7.1.1]{Ka} and ($G_{l-1}.1$)-($G_{l-1}.2$).
\end{proof}

\begin{lem}
For $\zeta\in Q^-(l-1)$, $b\in \cL_\Z(\infty)_-/q\cL_\Z(\infty)_-$, and $\lambda\in P^+$,
\[G(b) v^+_\lambda = G_\lambda(\hat \wp_\lambda b).\]
\end{lem}

\begin{lem}
For $\zeta\in Q^-(l-1)$ and $b\in \cL_\Z(\infty)_-/q\cL_\Z(\infty)_-$, we have
\[\bar{G(b)}=G(b).\]
\end{lem}

\begin{proof}
Set $Q=\frac{G(b)-\bar{G(b)}}{\pi q  - q^{-1}}$.  Then
$Q\in (\UUZ^-_-)^{\zeta}\cap q\cL_\Z(\infty)_-\cap \bar{\cL_\Z(\infty)_-}$, and
hence $Q=0$.
\end{proof}

Note that a super counterpart of the proof of \cite[Lemma~7.5.1]{Ka} that ``$G_i(b) =G_j(b)$" requires
the validity of  Proposition~\ref{prop:L*=L}.
The remaining components of the inductive proof of ($G_{l}.1$)-($G_{l}.3$)
proceed just as in \cite[\S7.4-7.5]{Ka}.

Now since ($G_{l}.1$)-($G_{l}.3$) hold in the cases $\diamond \in \{+,-\}$, we can lift them to the level of a formal parameter $\pi$,
and so ($G_{l}.1$)-($G_{l}.3$) for $\diamond =\emptyset$ also hold. Alternatively, one could go through Kashiwara's arguments
in the setting of $\diamond =\emptyset$ directly.

We summarize the main theorem on canonical bases.

\begin{thm}
Let $\diamond \in \{\emptyset, +, -\}$.
Then
\begin{enumerate}
\item
$\{G(b) \mid b \in B(\infty)\}$ forms a bar-invariant $\pi$-basis for $\UUZ^-_\diamond$.

\item
For every $\lambda\in P^+$,
$G(b) v^+_\lambda = G_\lambda(\hat \wp_\lambda b).$ Moreover,
$\{G_\la(b) \mid b \in B(\la)\}$ forms a bar-invariant $\pi$-basis for $V_\Z(\la)_\diamond$.
\end{enumerate}
\end{thm}

The maximal variants $\{G(b)\} \cup \{\pi G(b)\}$ and $\{G_\la(b)\} \cup \{\pi G_\la(b)\}$ of the bases in the above theorem
would be a signed basis
in a more conventional sense. The bases in the above theorem (or these maximal variants) will be referred to as canonical bases.

Our canonical basis is a $\pi$-basis in the sense of Definition~\ref{dfn:pibasis}, but not a genuine basis in general.
We do not regard this as a defect of our construction though
as this is completely natural from the viewpoint of categorification (\cite{HW}): $\pi$ corresponds to ``spin" (i.e., a parity shift functor $\Pi$),
 each (projective) indecomposable module
$M$ comes from two ``spin states" $\{M, \Pi M\}$, and there is no preferred choice among $M$ and $\Pi M$ a priori.

\begin{example}
Assume that $I_1$ contains $i,j$ such that $a_{ij}=a_{ji}=0$. Then
$F_iF_j=\pi F_jF_i$. Both $F_iF_j$ and $F_jF_i$ are canonical basis elements in $\UU^-$, and there is no preferred choice among the two.
\end{example}


\end{document}